\documentclass[11pt]{article}
\usepackage{amsfonts, amsthm, amssymb,sectsty,latexsym,textcomp}
\usepackage{graphics}
\usepackage{fullpage}
\usepackage{xcolor}
\usepackage{soul}
\usepackage{amsmath}
\usepackage{tikz}
\usepackage{tikz-cd}
\usepackage{forest}
\usepackage{subcaption}
\usepackage{float}	

\usetikzlibrary{positioning,calc,matrix,shapes.multipart}
\usetikzlibrary{arrows}

\numberwithin{equation}{section}

\title{Cellulose Biodegradation Models\SP; An Example of Cooperative Interactions in Structured Populations}

\author{Pierre-Emmanuel Jabin\thanks{Department of Mathematics,
University of Maryland, College Park, pjabin@umd.edu} \;\; Alexey Miroshnikov\thanks {Department of Mathematics, University of
California, Los Angeles, amiroshn@gmail.com} \;\;  Robin Young \thanks {Department of Mathematics and Statistics, University of
Massachusetts Amherst, young@math.umass.edu}}

\def\del{\partial}

\newcommand{\RR}{\mathbb{R}}
\newcommand{\eps}{\varepsilon}

\def\clap#1{\hbox to 0pt{\hss#1\hss}}

\newcommand{\SP}{\hspace{1pt}}

\newtheorem{theorem}{Theorem}

\newtheorem{lemma}{Lemma}
\newtheorem{remark}{Remark}

\numberwithin{remark}{section}
\numberwithin{theorem}{section}
\numberwithin{lemma}{section}
\newcommand{\cmt}[1]{#1}  %%  {\color{blue} #1}}

\newcommand{\ssst}[1]{{\scriptscriptstyle\it #1}}

\date{}

\begin{document}

\maketitle

\begin{abstract}
We introduce various models for cellulose bio-degradation by micro-organisms. Those models rely on complex chemical mechanisms, involve the structure of the cellulose chains and are allowed to depend on the phenotypical traits of the population of micro-organisms. We then use the corresponding models in the context of multiple-trait populations. This leads to classical, logistic type, reproduction rates limiting the growth of large populations but also, and more surprisingly,  limiting the growth of populations which are too small in a manner similar to the effects seen in populations requiring cooperative interactions (or sexual reproduction).
This study thus offers a striking example of how some mechanisms resembling {\em cooperation} can occur in structured biological populations, even in the absence of any actual cooperation.
\end{abstract}

\tableofcontents

\section{Introduction}
The goal of this article is to derive models for structured populations of micro-organisms living off cellulose degradation. Our first step is to study the mechanisms by which some micro-organisms can use cellulose. The full process is obviously complex and we have to abstract its most important features. This gives us a hierarchy of models, depending on the level of simplification that one desires.

The second step is to couple those models with the population dynamics of the corresponding micro-organisms. While the mechanism of bio-degradation that we consider is similar for each species of micro-organisms, we allow for some variability from one species to another, in the enzymes involved for instance. This leads to a population structured by a phenomenological trait that describes the exact path of bio-degradation.

As the amount of cellulose is limited, the total growth of the population is, unsurprisingly, limited as well. More interesting are the effects when the total population or the population in a given species is small. The model does not include any actual cooperation between micro-organisms but as the bio-degradation occurs in several steps, the process is nonlinear in the population size even if cellulose is abundant. This puts small populations at a disadvantage, introducing an effect similar to classical cooperation.

%In this article we study evolutionary dynamics in biological populations structured by a parameter describing a biological, physiological or ecological characteristic of the individuals. When this characteristic is inherent to the individual, it is called a {\it trait} \cite{Perthame07}. The objective of the present work is to study {\it cooperation} among the individuals in certain biological populations. Roughly speaking, such phenomena are reflected by a population growth rate which is super-linear with respect to the population size.  Typical models describing the population dynamics take only competition into account neglecting the fact that for relatively small populations cooperation may play an important role (as, for instance, small populations may survive for longer periods of time). In the present paper we develop models describing {\em cellulose biodegradation} processes (in which cooperation may potentially be present); we study the case of multiple-trait populations and investigate continuous-trait models that appear in the limit. %The mechanisms that initiate cooperation in these models have been discussed extensively with a microbiologist J.~Phillips (University of Leuven, Belgium and University of Massachusetts, Amherst).

\par\medskip

\noindent{\bf Cellulose Bio-degradation. Mechanisms and Models.}
Cellulose is the structural component of many plants and is therefore
the most abundantly produced bio-polymer; it is a homo-polymer
consisting of a vast number of glucose units. The most important feature
of cellulose as a substrate is its {\em insolubility}. As such,
bacterial and fungal degradation of cellulose, (e.g.\;by fungi {\it
  Trichoderma reesei} or bacteria {\it Clostridium thermocellum}),
occurs exocellularly. The products of cellulose hydrolysis are
available as carbon and energy sources for microbes that inhabit
environments in which cellulose is biodegraded \cite{LESH95, LYND02}.

\par

In this work we model cellulose bio-degradation as a multiple-step
process, reflecting realistic mechanisms described in \cite{LYND02}.
Let $\varrho(t)$ denote the mass of cellulose. 
The biodegrading microorganism is unable to consume (degrade) the cellulose $\varrho$ directly. Instead, the individuals produce two enzyme complexes
$e_{1}(t) $ and $ e_{2}(t)$ that act in a two-stage process.

During the first stage, the (endoglucanase) enzyme $e_{1}$ weakens
cellulose fibers in $\varrho$: that is, it randomly cuts the fibers,
creating the so-called reducing and non-reducing ends which serve as
landing sites for the (exoglucanase) enzyme $e_{2}$.
During the second stage, the enzyme $e_{2}$ locates a landing site and attaches itself to it.  Once attached, it
cleaves off cellobiose (a major energy source for the microorganisms)
from the chain of polysaccharides.  Some portion $\theta_p \in [0,1]$ of
cellobiose is consumed directly by the microorganism that produced the
enzymes, and the rest is available for other individual microorganisms
in the population due to diffusion. The above mechanisms can be viewed
as follows:
\begin{center}
Growth of micro-organisms $+$ influx of cellulose
$\varrho(t)$\\
$\downarrow$\\
Production of enzyme complexes $e_{1}(t)$ and  $e_{2}(t)$\\
$\downarrow$\\
Weakening of $\varrho(t)$ by $e_{1}(t)$\\
$\downarrow$\\
Production of cellobiose $p(t)$ by $e_{2}(t)$ acting on
$\varrho(t) $.
\end{center}
{The last two steps in the above diagram constitute the so-called cleaving mechanism. In our work we present two different cleaving mechanisms that differ in complexity (see Section \ref{sec:struct}). In Cleaving Mechanism~1 the enzyme $e_2$ that cuts off cellobiose units has two states, `attached to' and `detached from' cellulose. In Cleaving Mechanism 2, however, the enzyme $e_2$ is always detached. In this mechanism cleaving happens instantaneously once the enzyme $e_2$ finds a spot on the cellulose where it is able to cut off cellobiose.
}

\par

In the present work we develop several models of varying complexity
which incorporate these mechanisms.  Even in the simplest model the aforementioned cascade of events produces a {\it cooperative effect}, which appears due to the fact that the cellobiose units cleaved off by the enzyme of one microorganism are available for consumption by other individuals
located nearby.  Mathematically, these effects are encoded in the
reproduction rate $B(n)$ of the population $n$.  In particular, for
small populations, the population size $n(t)$ turns out to behave as
\begin{equation*}
\del_t n(\cdot , t)  \sim n \big( B(n) - d \big) \quad \mbox{with}
\quad B(n) \sim C n^2\, \quad \mbox{when} \quad
 n \ll \bar n,
\end{equation*}
where $\bar n$ is a critical threshold.

In general, the population includes various species of
micro-organisms. In that case, the exact enzyme complexes used may
change across species.  We represent the different species $x_j$ by traits
$j\in\{1,\dots, M\}$, where the population with trait $j$ uses
enzymes complexes denoted $e_{1,j}$ and $e_{2,j}$.  {This can be
  included in a more general framework by considering continuous
  traits $x$ with sub-populations $n(t,x)$ with a model of the form
\begin{equation}\label{ctrait}
  \del_t n(x,t)  = \Big( B[n](x,t) - d(x) \Big)n(x,t)\,,
\end{equation}
where $B[n]$ is now an integral operator; see \eqref{Bcont}
for the precise formula.
The case of discrete traits is included in this framework by
taking $n(t,x)=\sum_j n_j(t)\,\delta_{x_j}(x)$ with $n_j(t)$ the
population of individuals with trait $x_j$.

\smallskip

\cmt{ \noindent{\bf Hierarchy of Models.} Let us give a brief description of the models developed in our work and their relations. 
\[
 \overbrace{\text{($N$-$S$-model)} \to \text{($S$-model)} }^{\text{Mechanism 1}} \to \overbrace{\text{($T$-model)} \leftarrow \text{(multiple-trait $T$-model)}}^{\text{Mechanism 2}}.
\]
The most complex model that explicitly takes the kinetics of enzymatic
reactions into account is the $N$-$S$-model. We monitor the evolution
of the cellulose chains $N$, structured by polymer length and the
total number of landing sites where the enzyme $e_2$ can be
attached. This model incorporates cleaving Mechanism~1, in which the
enzyme $e_2$ is either `attached' ($e_{2A}$) or `detached'
($e_{2D}$). In addition, the model tracks the evolution of unoccupied
landing sites $S$ and the total number of landing sites $T$; these two
variables are related via $T = S + {e_{2A}}$. Under the assumption
that the cleaving rates are independent of the polymer structure, the
$N$-$S$-model reduces to the $S$-model. The next reduction occurs when
cleaving Mechanism 1 is replaced with Mechanism 2 in which the two stages of cleaving are combined into one. Here, the enzyme $e_2$ is always detached (the attachment and cleaving of cellobiose occur instantaneously) while the number of unoccupied landing sites and total number of sites coincide $T=S$, leading to the $T$-model. Though this is the simplest model it still captures the basic features of cellulose biodegradation. For this reason we extend it to a multiple-trait $T$-model which allows for species structured by a parameter. Finally to study more specifically cooperative effects we modify the above models to take only time scales on which the population changes into account (see Section \ref{sec:coop}).
}

This framework has some interest for the analysis even when it is only
applied for a finite number of traits as is the case here. Indeed our
traits correspond to possible enzyme complexes and while some of them
may still be unknown, we do not expect their number to be very
large. We also refer to \cite{Gyllenberg} for the reason that in general one can only expect a finite (if possibly very large) number of traits at the ecological equilibrium, Evolutionarily Stable Strategy or ESS.}

Our model therefore resembles systems of population dynamics, see
\cite{Perthame07} for instance.  However, in contrast to many of those
systems, the cellulose bio-degradation process leads to both
competition between individuals and species (for the resource) and
cooperative interactions.  This occurs at the interspecies
level as the byproduct of the process, cellobiose, is the same
independent of the enzyme complexes involved, and can benefit any
individual in the population (and not only individuals using the same
complexes).
Cooperation also occurs specifically within each species (or between
species that are close enough).  This follows from the fact that an
individual with similar enough enzyme complexes can use {a landing site}  created in the cellulose by the endoglucanase enzyme complex of another individual. The mathematical models developed in our work thus
lead to different (and hopefully improved) phenomenological results for small populations (in deterministic models) since cooperation significantly affects the dynamics, as discussed in more detail below. Those differences of behaviour for small populations may not impact the final ESS, but they are important in the transitory regime, in particular in the presence of mutations.

\cmt{From the ecological point of view, an important conclusion of our modeling is that as soon as minimally complex biochemical processes are involved, one cannot simply interpret the relation between the individuals in the population (or in different sub-populations) as competitive or cooperative. When the population is very large, the interaction looks competitive because resources are limited. On the other hand when the population is very small, the interaction seems to be cooperative as the limitation on growth mostly comes from the ability of the individuals to process efficiently the several steps of the biochemical process. But this is only a caricature of the actual interaction which cannot be reduced to pure cooperation or competition.}

\smallskip

 \cmt{\noindent{\bf Tail issue in deterministic selection dynamics.} The cooperative nature of the interaction when populations are small also becomes important for deterministic selection models. As sub-populations may grow or decrease exponentially, there typically are several orders of magnitude between the large populations of dominant traits and smaller ones. This poses an acute problem for modeling as one would need to use deterministic model for the larger populations. But such deterministic equations do not adequately capture the stochastic nature of the dynamics of smaller populations.
This problem can however be alleviated if such small populations go extinct since accurately modeling their behavior loses its relevance. This is precisely what happens if cooperation is needed when the population goes below a certain threshold: The birth rate of a small population is then necessarily too small ensuring a negative reproduction rate and extinction. We further discuss this phenomenon in the appendix.}

\par

\smallskip

\cmt{
\noindent{\bf Structure of the paper.} In Section \ref{sec:struct}, we
set notation and, following the seminal article \cite{LYND02},
introduce basic processes and mechanisms that constitute cellulose
biodegradation. In Section \ref{sec:models}, we develop a series of
mathematical models, that differ in complexity, for the coevolution
between microorganisms which consume cellulose and the cellulose
chains. In Section \ref{sec:coop}, we carry out a qualitative analysis
of the models developed in Section \ref{sec:models} under the
assumption that cellulose dynamics and enzymatic reactions occur on a
faster time scale than the dynamics of the microbial population. In
Section \ref{sec:conttrait}, the multiple-trait model, developed
earlier in Section \ref{sec:models}, is extended to a continuous-trait
model.  We carry out some numerical experiments in Section
\ref{sec:num_exp}, in which we compare the models and demonstrate that
the $T$-model can be regarded as a limit of the $S$-model.  Finally,
we discuss the tail issue in more detail in the appendix.
}

\section{Cellulose bio-degradation: structure and mechanisms}
\label{sec:struct}

\subsection{Cellulose structure and enzyme systems}

Cellulose is the most abundantly produced bio-polymer.  It is a
homo-polymer consisting of glucose units joined by $\beta$-$1,4$
bonds.  In secondary walls of plants, the size of cellulose molecules
(degree of polymerization) varies from seven thousand to fourteen
thousand glucose moieties per molecule.  Cellulose molecules are
strongly associated through inter- and intra-molecular
hydrogen-binding and van der Waals forces that result in the formation
of microfibrils, which in turn form fibrils.  Cellulose molecules are
oriented in parallel, with reducing ends of adjacent glucan chains
located at the same end of a micro-fibril.  These molecules form
highly ordered crystalline domains interspersed with more disordered,
amorphous regions.  Although cellulose forms a distinct crystalline
structure, cellulose fibers in nature are not purely crystalline.  The
degree of crystallinity varies from purely crystalline to purely
amorphous.

\par

To degrade plant cell material, microorganisms produce multiple
enzymes known as enzyme systems \cite{LYND02}.  For microorganisms to
hydrolyze and metabolize insoluble cellulose, extra-cellular cellulases
(degradation enzymes) must be produced that are either {\it free} or
{\it cell associated}.  Microorganisms have adapted different
approaches to effectively hydrolyze cellulose, naturally occurring in
insoluble particles.  Cellulosic filamentous fungi (and some types of
aerobic bacteria) have the ability to penetrate cellulosic substrates
through hyphal extensions, thus often presenting their {\it free}
cellulase systems in confined cavities within cellulosic particles.
In contrast, anaerobic bacteria lack the ability to effectively
penetrate cellulosic material and perhaps had to find alternative
mechanisms for degrading cellulose.  This led to the development of
{\it complexed} cellulase systems (called cellulosomes) which position
cellulase producing cells at the site of hydrolysis, as observed for
clostridia and ruminal bacteria.

\par

Overall there are three major components of cellulase systems: $(i)$
endoglucanases, which randomly hydrolyze $\beta$-$1,4$ bonds within
cellulose molecules, thereby producing reducing and non-reducing ends;
$(ii)$ exoglucanases, which \cmt{liberate (cleave off) either glucose (glucanohydrolases) or cellobiose (cellobiohydrolase) that serve as major products} from \cmt{reducing or non-reducing ends of cellulose polysaccharide chains;} and $(iii)$ $\beta-$glucosidases which hydrolyze cellobiose yielding glucose (the major product of cellulose hydrolysis used by microorganisms as energy source). For details see \cite{LYND02}.

\subsection{Quantities to monitor and two basic bio-mechanisms} 

{We first consider the case of populations with one trait.  In our
  analysis we let $n=n(t)$ denote the total \cmt{number} of the
  microorganism that degrades cellulose, while  the total \cmt{number} of
  {endoglucanases and exoglucanases} produced by the microorganism are
  denoted by $e_{1}=e_1(t)$ and $e_{2}=e_2(t)$, respectively.

\par

We view cellulose as a crystalline conglomerate of fibers (chains of
polysaccharides).  According to our discussion above, during the first
stage of the degradation the endoglucanase enzyme $e_1$ weakens
fibers, which means that $e_1$ randomly cuts the fibers by creating
reducing and non-reducing ends that then serve as landing sites for
exoglucanases $e_2$.  Viewing cellulose as a three-dimensional
structure, one can imagine that structure with punctures or cuts after
the first stage.  It is still the same cellulose but with more `cuts'
that serve as landing sites for the exoglucanase enzymes $e_2$.

\par

\cmt{ There are two different complexes of exoglucanases that are
  able to cleave off (liberate) cellobiose units from the cellulose
  chains. These are exoglucanase CBHI and exoglucanase CBHII.  The
  enzymes of the first type are able to attach to reducing ends, and
  the second to non-reducing ends (see \cite[p. 512]{LYND02}).  In our
  models, for simplicity, we do not distinguish between the two
  types; both are represented by $e_2$.  For this reason, we do not
  differentiate between reducing and non-reducing ends, and call them
  `landing sites' instead.  We add that each landing site
  (created by the endoglucanase $e_1$) always contains one reducing
  end and one non-reducing end.  In our model, however, we allow each
  landing site to host only one enzyme.}   Once the exoglucanase $e_2$ lands on the {chain}, it cleaves off cellobiose from the chain of polysaccharides.  In this treatment we do not consider the third type of enzyme ($\beta-$glucosidase) and treat cellobiose as the major product of
degradation.  We instead assume that some portion $\theta_p \in [0,1]$
of cellobiose is consumed by the microorganism that produced it, while
the rest diffuses and is freely available for general consumption.  We
let $p(t)$ denote the total of the freely available cellobiose.  }

%\subsection{Cleaving mechanisms}

\par\medskip

Thus in our model, there are two main stages in which the exoglucanase
$e_2$ produces cellobiose $p(t)$: first, the enzyme locates {a
  landing site} and attaches
itself to the chain there; next, it keeps cleaving off cellobiose
units at a certain rate until it either disintegrates or detaches from
the chain.  This leads to two basic modeling approaches, which we call
cleaving mechanisms.

\subsubsection{Cleaving Mechanism 1 }

Since the time spent by an individual exoglucanase enzyme $e_2$
locating a 	{landing site}  may differ
significantly from the time it is attached to the landing site, it is
useful to consider two states for exoglucanase, namely {\it detached}
and {\it attached} states.  In the first mechanism we distinguish
them, letting $e_{2D}(t)$ represent the amount of detached $e_2$
(which may wander freely or on a leash, that is attached to a
bacterial cell wall), and $e_{2A}(t)$ the amount attached to
a landing site.
\cmt{Denoting the total number of landing sites by $T(t)$ and the
number of {\em unoccupied} landing sites by $S(t)$, it then follows that}
\begin{equation}
\label{TSrel} T(t)=S(t)+{e}_{2A}(t)\,.
\end{equation}

We suppose that, at any moment of time, unoccupied spots $S(t)$ become
occupied (or attacked) by the detached enzymes $e_{2D}$ at a certain
rate to be described. Next, we assume that an individual attached
enzyme $e_{2A}$ cleaves off cellobiose units from the {cellulose
  chain}, again with a given rate. Also, we assume
that some proportion of the (attached) enzyme $e_{2A}$ detaches from
{a landing site}, and that some fraction $\theta_r\in[0,1]$ of those {sites} \cmt{become unavailable for landing, that is, the landing sites are destroyed.}

\subsubsection{Cleaving Mechanism 2}

The second mechanism is somewhat simplified.  It may be used to
describe complex cellulases where exoglucanases are not entirely free
(they are attached to bacterial cell walls, and once bacteria leaves
the spot the enzyme becomes detached {from the landing site} as well).  Here we suppose that at any moment of time, all existing $T(t)$ {landing sites} are available for an attack by the enzyme $e_2$.  The {landing sites}
$T(t)$ are attacked with a certain rate $b(T(t))$
and a certain (average) amount $q>0$ of cellobiose units is cleaved
off by each individual enzyme $e_2$, after which the enzyme $e_2$
detaches itself.  We view such an attack as instantaneous.  Thus,
after such an (instantaneous) attack, all $T(t)$ {landing sites}
 are again unoccupied.  We will assume that after
the attack a certain portion $\theta_r\in[0,1]$ of the (attacked)
{landing sites} \cmt{become unavailable, that
  is, destroyed}.  In that scenario the two
processes, finding a {landing site} and
cleaving off cellobiose, are lumped together (with a hidden assumption
that enzymes cannot remain attached to a chain for a very long time).

\begin{remark} \rm This first mechanism is more realistic since it
  takes into account time spent by the enzyme on {a site}.  This mechanism can be employed for modeling
  systems where both non-complex and complex cellulases are present.
\end{remark}

\begin{remark}\rm
  In the second mechanism the {landing sites} $T(t)$ serve as "prey" and $e_2$ as "predators", with one difference: the enzymes attack the prey, use it and leave it alone. After an attack only a certain proportion of the sites is destroyed, while the rest is still usable.
\end{remark}

\section{Cleaving Models}

\label{sec:models}

\cmt{We now develop several models of varying complexity to describe
  cellulose biodegradation.  We begin with a one-trait model
  describing cleaving mechanism 1, which retains the most detailed
  cellulose structure.  This $N$-$S$ system is too cumbersome to be
  effectively analyzed, so we reduce it by removing the detailed
  cellulose structure, to obtain the $S$ system, a set of 7 equations
  which models cleaving mechanism 1.  We then modify the $S$ system to
  model cleaving mechanism 2, which yields the $T$ system, consisting
  of 6 equations.  Finally, we adapt the simplest of these, the
  $T$ system, to a multiple trait model, in which there can be several
  species of microorganism consuming the same cellulose.}

\subsection{$N$-$S$-model} \label{NSmodel}

We first consider a model in which we monitor groups of cellulose
chains consisting of $l \geq 1$ cellobiose units; in that case we say
that the chain has length $l$.  This allows us to develop a
fundamental model incorporating cleaving mechanism 1.

\par

\noindent{\em Assumptions and notation.}\quad Cellulose chains may
have different topological configurations: they could be linear or
rectangular (when fibers are embedded in a lignin matrix) or they
could have a random three-dimensional structure.  Monitoring the
topology increases the complexity, but it does not provide a better
tool for studying the population dynamics.  After all, it is the
number of landing spots that matters rather than the configuration of
the cellulose chains.  Thus we make no assumption about the
configuration and monitor only the length of its constituent pieces.
Another assumption we make is that the enzyme $e_1$ produces a landing
site without physically cutting the chain. This assumption decreases
complexity in the model while it does not change the dynamics.
Indeed, if we allow $e_1$ to physically cut a chain in the model, then
the chain could be split into two parts when $e_1$ acts.  In that
case, the number of landing sites would be the same as in the scenario
when the landing site is created without a physical cut. Finally, we
impose the requirement that at most one landing site per unit of
cellobiose is allowed; this reflects the fact that cellulose chains
represent systems of discrete units.

\par

We say that a cellulose chain is in the $(l,i)$-state, or is an
$(l,i)$-chain, if it has length $l$ (so consists of $l \geq 1$
cellobiose units) and $i \in \{0,1,\dots,l\}$ landing sites
(previously made by enzyme $e_1$), and let $N^{l,i}(t)$ denote the
number of $(l,i)$-chains.  Recall that $p$ denotes the number of
cellobiose units available for general consumption, $e_1$ denotes the
number of endoglucanase enzymes, which produce landing sites, and
$e_{2D}$ denotes the number of detached exoglucanase enzymes.  We
refine the attached exoglucanase to refer to those enzymes attached to
chains in the $(l,i)$-state by $e_{2A}^{l,i}$, where the \cmt{superscripted}
indices $(l,i)$ are in the set
\begin{equation}\label{ILdef}
\begin{aligned}
(l,i) \in I_L := \{ (\tilde{l},\tilde{i}) \in \mathbb{Z} \times \mathbb{Z}:  1 \leq \tilde{l} \leq L,\, 0\leq \tilde{i}\leq \tilde{l}\}\,.
\end{aligned}
\end{equation}
Here $L$ stands for the maximal number of cellobiose units in
cellulose chains.  We also use the convention that
\begin{equation}\label{convind}
\begin{aligned}
N^{l,i}     &\equiv 0 \quad \mbox{if}  \quad
           (l,i) \not\in I_L,\\
e^{l,i}_{2A} &\equiv 0 \quad \mbox{if}  \quad
           (l,i) \not\in I_L \quad \mbox{or} \quad i=0\,.
\end{aligned}
\end{equation}

\par\medskip

\noindent{\em Enzyme dynamics.}  We assume that the rates of
production of the enzymes $e_1$, $e_2$ by the microorganism and their
degradation rates are fixed.  The enzymes $e_1$ and $e_2$ are
catalyzers which stay in the system as long as they ``live''.  Then
$e_1$ satisfies
\begin{equation}\label{enz1-NS}
\del_t \SP {e}_1(t) = {b}_1 \SP n(t) - d_1  {e}_1(t) \quad  \mbox{with}
\quad b_1,\ d_1>0\,.
\end{equation}

Next, the number of landing sites on $(l,i)$-chains is $T^{l,i}=i\,
N^{l,i}$, so the number of unoccupied landing sites $S^{l,i}(t)$ is
\begin{equation}
\label{Sloc-NS}
  S^{l,i}(t) =  T^{l,i}(t) - {{e}_{2A}^{l,i}(t)}
            = i\, N^{l,i}(t) - {{e}_{2A}^{l,i}(t)}.
\end{equation}
Neglecting saturation effects, we suppose that unoccupied sites
$S^{l,i}$ are attacked by $e_{2D}$ with the rate $\beta^{l,i}S^{l,i}$.
Also, we assume that enzyme $e_2$ located on a chain in $(l,i)$-state
(randomly) detaches from the chain with rate $\sigma^{l,i}>0$.  We let
$\gamma_{r}^{ l,i}>0$ denote the decay rate of an individual landing
site (whether it is occupied or not) and assume that if an occupied {landing site} degrades then the attached enzyme $e_{2A}$ disintegrates
together with it.  This leads to the following set of equations that
monitor the dynamics of the enzyme $e_2$:
\begin{equation}\label{enz2-NS}
\begin{aligned}
\del_t {e}_{2D}(t)  & = {b}_2 \SP n(t)  - \sum_{l,i} \beta^{l,i} S^{l,i}(t) \;{e}_{2D}(t) + \sum_{l,i} \sigma^{l,i}\SP {e}_{2A}^{l,i}(t) - d_{2D} \SP {e}_{2D}(t)\\
 & = {b}_2 \SP n(t)  - \sum_{l,i} \bigg[\beta^{l,i}  \Big( i N^{l,i}(t) - {{e}_{2A}^{l,i}(t)}\Big) {e}_{2D}(t) - \sigma^{l,i} \SP {e}_{2A}^{l,i}(t)\bigg] - d_{2D} \SP {e}_{2D}(t)\\
\del_t  {e}_{2A}^{l,i}(t) & =   \beta^{l,i} S^{l,i}(t) \; {e}_{2D}(t) - \Big(\sigma^{l,i} + d_{2A}^{\,l,i} + \gamma_{r}^{ l,i}\Big)\,{e}_{2A}^{l,i} \\
& = \beta^{l,i}\Big( i N^{l,i}(t) - {{e}_{2A}^{l,i}(t)}\Big) {e}_{2D}(t)  - \Big(\sigma^{l,i} + d_{2A}^{\,l,i} + \gamma_{r}^{ l,i}\Big) {e}_{2A}^{l,i}
\end{aligned}
\end{equation}
where $d_{2D}>0$ and $d_{2A}^{\,l,i}> 0$ are the degradation rates of
$e_{2D}$ and ${e_{2A}^{l,i}}$, respectively, and $b_2>0$ is the
production rate of $e_{2}$, which equals that of $e_{2D}$.

\par\medskip

\noindent{\em Chain dynamics.}  We neglect saturation effects and
assume that the rate with which enzymes $e_1$ produce landing sites on
chains in the $(l,i)$-state is
\begin{equation}
\label{transcut-NS}
  \alpha^{l,i} (l-i)  N^{l,i}(t) {{e}_1(t)} =
   \widehat{\alpha}^{l,i} \SP N^{l,i}(t)\SP  {e}_1(t)\,,\quad
   \mbox{with} \quad
  \widehat{\alpha}^{l,i} = \alpha^{l,i} {(l-i)}
\end{equation}
where the multiplier $(l-i)$ reflects the requirement that only one
landing site per cellobiose unit is allowed.  We assume that a
freshly made landing site cannot be instantaneously occupied and
different cuts don't occur simultaneously, so that \eqref{transcut-NS} is the transition rate of states
$(l,i) \to (l,i+1)$ due to the action of $e_1$.

\par

Recall that $\gamma_{r}^{ l,i}>0$ is the decay rate of an individual
landing site on a chain in $(l,i)$-state, so that the rate of transition $(l,i) \to
(l,i-1)$ due to degradation of landing sites is
\begin{equation}
\label{translsd-NS} \widehat{\gamma}^{l,i} N^{l,i}(t), \quad
\mbox{with} \quad \widehat{\gamma}^{l,i}=i \gamma_{r}^{ l,i}\,.
\end{equation}

\par

\cmt{

Let $q^{l,i}>0$ denote the rate of production of cellobiose by an
individual enzyme attached to an $(l,i)$-chain. Then, the total rate of
cellobiose production by the enzymes $e^{l,i}_{2A}$ is
\begin{equation}
\label{transclv-NS}
q^{l,i} {{e}_{2A}^{l,i}(t)}.
   % = \widehat{q}^{l,i} \SP m_2 {e}_{2A}^{l,i}(t) \,, \quad \mbox{with} \quad
   %   \widehat{q}^{l,i}= \frac{q_{\SP l,i} }{m_2}\,,
\end{equation}
The rate $q^{l,i}$ can be expressed as 
\[
q^{l,i}=c^{l,i}+p^{l,i}
\]
where, $c^{l,i}$ is the rate of cleaving that results in the transition $(l,i) \to (l-1,i)$ and $p^{l,i}$ is the rate of cleaving that results in the transition $(l,i)\to(l-1,i-1)$. The latter transition occurs when an enzyme $e_{2A}^{l,i}$ cleaves off a cellobiose unit and when moving along the chain reaches the next cellobiose unit that also contains a landing site (this results in the decrease of landing sites by one on the given chain). We note that $c^{l,l}=c^{l,0}=0$ and $p^{l,0}=p^{l,1}=0$ for any $l$.
}

% Let $q^{l,i}>0$ denote the rate of production of cellobiose by an
% individual enzyme attached to an $(l,i)$-chain.  Then the total rate of
% cellobiose production by the enzymes $e^{l,i}_{2A}$ is
% \begin{equation}
% \label{transclv-NS}
%    q^{l,i} {{e}_{2A}^{l,i}(t)}
%    % = \widehat{q}^{l,i} \SP m_2 {e}_{2A}^{l,i}(t) \,, \quad \mbox{with} \quad
%    %   \widehat{q}^{l,i}= \frac{q_{\SP l,i} }{m_2}\,,
% \end{equation}
% % and this is the rate of transition $(l,i) \to (l-1,i)$ due to the cleaving activity of $e_2$.
% \cmt{Note that $(1-\delta_{li}) q^{l,i} {{e}_{2A}^{l,i}}$ and
% $\delta_{li} q^{l,i} {{e}_{2A}^{l,i}}$, where $\delta_{li}$ denotes
% Kronecker delta, are the rates of transition $(l,i) \to (l-1,i)$ and
% $(l,i) \to (l-1,i-1)$, respectively,	 due to the cleaving activity
% of $e_2$}. 

Let $\theta_r>0$ denote the proportion of landing sites that {get
  destroyed} after $e^{l,i}_{2A}$ detaches from a
chain or dies. This contributes to the transition $(l,i) \to (l,i-1)$
and the corresponding rate is
\begin{equation}\label{transdun-NS}
  \widehat{\theta}^{l,i} \, {e}_{2A}^{l,i}(t)\,
   \quad \mbox{with} \quad
  \widehat{\theta}^{l,i}=
    \theta_r  (\SP \sigma^{l,i}+d_{2A}^{\,l,i})
\end{equation}

Combining \eqref{transcut-NS}--\eqref{transdun-NS} we obtain the
equations that monitor the dynamics of $N^{l,i}(t)$, namely
\cmt{
\begin{equation}\label{Neqn_NS_prelim}
\begin{aligned}
  \del_t N^{l,i}(t) = r^{l,i} & +
\Big({\widehat{\alpha}^{l,i-1}}N^{l,i-1}(t)-{\widehat{\alpha}^{l,i}}
N^{l,i}(t)\Big) {e}_1(t)
 + \Big(\widehat{\gamma}^{l,i+1} N^{l,i+1}(t) - \widehat{\gamma}^{l,i}
 N^{l,i}(t)\Big) \\ &
+ \Big( c^{l+1,i} e_{2A}^{l+1,i} + p^{l+1,i+1} e_{2A}^{l+1, i+1} -(c^{l,i}+p^{l,i}) e_{2A}^{l,i} \Big) \\
&+\Big( \widehat{\theta}^{l,i+1} \SP {e}_{2A}^{l,i+1}(t) -
\widehat{\theta}^{l,i}\SP {e}_{2A}^{l,i}(t)\Big)
 - \gamma_{\varrho}^{l,i}N^{l,i}(t),
\end{aligned}
\end{equation}
}
where $r^{l,i}$ is the unit rate of production of cellulose, and
$\gamma_{\varrho}^{l,i}$ is the rate at which the cellulose naturally
decays or becomes unavailable to the microorganism (that is, decay not
directly attributable to the bacteria).  We assume for simplicity that
the cellulose provided by the environment has no {landing sites}, so
that $r^{l,i}=0$ for $i\geq 1$, and in particular, the sum
$\sum_ii\,r^{l,i}=0$.

\cmt{
 When polymer chains are long, that is $l$ is large, and the landing sites are spaced out, which happens when $i$ is small relative to $l$, the coefficients $p^{l,i}$ can be neglected. To simplify the equation \eqref{Neqn_NS_prelim} we drop  the coefficients $p^{l,i}$ except when $i=l$, corresponding to the boundary case when the number of sites $l$ equals to the number of cellobiose units. The coefficients $p^{l,l}$ cannot be dropped because cleaving cellobiose on the chain in state $(l,l)$ always leads to $(l-1,l-1)$. Moreover, dropping the coefficients $p^{l,l}$ would lead to a loss of conservation of the total cellulose in the system  when consumers are not present. This assumption leads to
\[
q^{l,i} = c^{l,i} \quad \text{for} \quad l \neq i \quad \text{and} \quad q^{l,l}=p^{l,l}
\]
and hence \eqref{Neqn_NS_prelim} becomes
\begin{equation}\label{Neqn-NS}
\begin{aligned}
  \del_t N^{l,i}(t) = r^{l,i} & +
\Big({\widehat{\alpha}^{l,i-1}}N^{l,i-1}(t)-{\widehat{\alpha}^{l,i}}
N^{l,i}(t)\Big) {e}_1(t)
 + \Big(\widehat{\gamma}^{l,i+1} N^{l,i+1}(t) - \widehat{\gamma}^{l,i}
 N^{l,i}(t)\Big) \\ &
+ \Big( {{q}^{l+1,i}} \SP {e}_{2A}^{l+1,i}(t) + \cmt{\delta_{li}q^{l+1,i+1}e_{2A}^{l+1,i+1}} -
{{q}^{l,i}}\SP {e}_{2A}^{l,i}(t)\Big)\\
&+\Big( \widehat{\theta}^{l,i+1} \SP {e}_{2A}^{l,i+1}(t) -
\widehat{\theta}^{l,i}\SP {e}_{2A}^{l,i}(t)\Big)
 - \gamma_{\varrho}^{l,i}N^{l,i}(t).
\end{aligned}
\end{equation}
where $\delta_{li}$ denotes the Kronecker delta.
}

\par\medskip

\noindent{\em Population dynamics.} Let $\theta_p\in[0,1]$ denote the
proportion of produced cellobiose that becomes available for everyone.
Then, using \eqref{transclv-NS}, the equations for the total amount
$p(t)$ of cellobiose available for everyone, and the total population
$n(t)$ of the microorganism are respectively
\begin{equation}\label{PNeqn-NS}
  \begin{aligned}
  \del_t p(t) &= \theta_{p}  \sum_{l,i} {q}^{l,i} {e}_{2A}^{l,i}(t) - \gamma\, n(t)\,p(t)-\gamma_p \,p(t),\quad\mbox{and}\\[2pt]
\del_t n(t) &=  \frac{\mu\, n(t)}{\bar{n}+n(t)} \bigg( \gamma\, p(t) +
(1-\theta_p) \sum_{l,i} q^{l,i} \,{e}_{2A}^{l,i}(t)\bigg)
  -\gamma_n\, n(t),
  \end{aligned}
\end{equation}
where $\gamma$ is the
consumption rate, $\mu$ is the conversion efficiency, and $\gamma_p$,
$\gamma_n$ are decay rates of $p$ and $n$, respectively.  Here $\bar
n$ represents a critical population threshold: if $n$ is large, $n =
O(\bar n)$, the growth depends only on the cellobiose supply, while if
$n$ is small, $n \ll \bar n$, the growth is linear but with small
growth rate, so the population is unlikely to survive (since
$\mu/\bar{n}<\gamma_n$).

\par\medskip

\cmt{
\noindent{\em Summary.} Figure \ref{transfig} shows possible states of the resource, state transitions and the rates at which these occur.

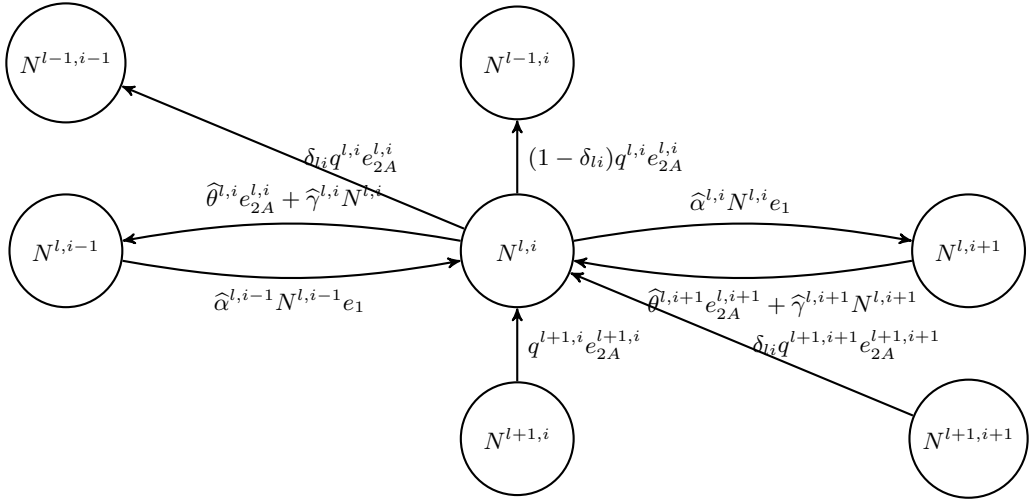
\begin{figure}[H]
\begin{center}
\begin{tikzpicture}[scale=1,->,>=stealth',auto,node distance=3cm,
  thick,main node/.style={circle,draw,font=\sffamily\Large\bfseries}]
%   [
% node distance=1cm and 1.5cm,
% arrow/.style={
%   ->,
%   >=latex,
%   shorten >= 3pt,
%   shorten <= 3pt,
% }
% ]
% \node[draw,minimum size=3cm,label={270:text}] (rect) {};
\node[circle,draw, minimum size = 1.5cm]
(circ_l_i) at (6,0)  {{\footnotesize $N^{l,i}$}};
\node[circle,draw, minimum size=1.5cm]
(circ_l_im1) at (0,0)  {{\footnotesize $N^{l,i-1}$}};
\node[circle,draw, minimum size=1.5cm]
(circ_l_ip1) at (12,0)  {{\footnotesize $N^{l,i+1}$}};
\node[circle,draw, minimum size=1.5cm]
(circ_lm1_i) at (6,+2.5)  {{\footnotesize $N^{l-1,i}$}};
\node[circle, draw, minimum size=1.5cm]
(circ_lm1_lm1) at (0,+2.5)  {{\footnotesize $N^{l-1,i-1}$}};
\node[circle, draw, minimum size=1.5cm]
(circ_lp1_i) at (6,-2.5)  {{\footnotesize $N^{l+1,i}$}};
\node[circle, draw, minimum size=1.5cm]
(circ_lp1_lp1) at (12,-2.5)  {{\footnotesize $N^{l+1,i+1}$}};

\draw [->] (circ_l_i) -- node[right] {\footnotesize$\delta_{li}q^{l,i}e_{2A}^{l,i}$} (circ_lm1_lm1);

\draw [->] (circ_l_i) -- node[right] {\footnotesize$(1-\delta_{li})q^{l,i}e_{2A}^{l,i}$} (circ_lm1_i);

\draw [->] (circ_l_i)  to [out=10, in=170] node[above] 
{\footnotesize${\widehat{\alpha}^{l,i}  N^{l,i}{e}_1 \,}$} (circ_l_ip1);

\draw [->] (circ_l_i) to [out=170, in=10] node[above]
{\footnotesize ${\;\widehat{\theta}^{l,i} {e}_{2A}^{l,i} + \widehat{\gamma}^{l,i}N^{l,i} }$\qquad} (circ_l_im1);

\draw [->] (circ_l_im1) to [out=-10, in=190] node[below] 
{\footnotesize${\widehat{\alpha}^{l,i-1}  N^{l,i-1}{e}_1 \,}$} (circ_l_i);

\draw [->] (circ_l_ip1) to [out=190, in=-10] node[below]
{\footnotesize \quad\qquad ${\;\widehat{\theta}^{l,i+1} {e}_{2A}^{l,i+1} + \widehat{\gamma}^{l,i+1}N^{l,i+1} }$} (circ_l_i);

\draw [->] (circ_lp1_i) -- node[right] {\footnotesize$q^{l+1,i}e_{2A}^{l+1,i}$} (circ_l_i);

\draw [->] (circ_lp1_lp1) -- node[right]{\footnotesize$\delta_{li} q^{l+1,i+1}e_{2A}^{l+1,i+1}$} (circ_l_i);

\end{tikzpicture}
\end{center}\caption{Transition of cellulose chains.}\label{transfig}
\end{figure} 
}

\begin{remark} \rm
  The processes of creating a landing site or cleaving
  off a cellobiose unit may depend on the configuration
  of the chain, the crystallinity of the cellulose as well as the
  lifetime of the enzymes.  Thus it is possible that within a given
  period of time (no matter how short) more than one landing site is
  created or two or more cellobiose units are cleaved off from the
  same chain.  For simplicity, we discount transitions other than
  $(l,i) \to (l,i+1)$ and $(l,i) \to (l-1,i)$, essentially assuming
  instantaneous transition.  This approach is justified provided that
  the number and size of chains is very large compared to the amount
  of enzymes $e_1$, $e_2$, and the likelihood that two landing sites
  are created or more than two cellobiose units are cleaved off from
  the same chain simultaneously (or during a short period of time) is
  extremely small.

Another way to justify this assumption is to consider a
time-continuous Poisson counting process, that corresponds to the
events of creating a landing site and/or cleaving off cellobiose.  Any
instance when a landing site is created (that is the moment it becomes
available for use by $e_2$) or cellobiose unit is cleaved off the
chain can be counted as an event. It is well-known that the
probability of two or more events happening instantaneously is zero
(in other words the probability that two events take place over the
time $\Delta t$ is $o(\Delta t)$; 
for details see \cite{LWL95,TLKAR98}.
\end{remark}

\subsection{Reduction to $S$-model for Cleaving Mechanism 1}

We now develop a simpler model by reducing the $N$-$S$-model,
consisting of \eqref{enz1-NS}, \eqref{enz2-NS}, \eqref{Neqn-NS} and
\eqref{PNeqn-NS}.  It is convenient to allow the indices $l$ and $i$
to run through all of $\mathbb{Z}$.  This augments the previously
defined system, but by choosing appropriate constants, we can ensure
that $N^{l,i}$ and $e_{2A}^{l,i}$ vanish for all times whenever
$l \leq 0$, $i < 0$, or $i>l$ and\cmt{, in addition, $e_{2A}^{l,0}=0$ for any $l$}.  We then define the quantities
\begin{equation}
\label{totqntdef}
\begin{aligned}
  \varrho(t) &= \sum_{l,i} l N^{l,i}(t),&\qquad &
  e_{2A}(t)=\sum_{l,i} e_{2A}^{l,i}(t), \\
  T(t)&=\sum_{l,i} T^{l,i}(t)=\sum_{l,i} iN^{l,i}(t), &&
  S(t)=\sum_{l,i}S^{l,i}(t),\\
 \end{aligned}
\end{equation}
which represent the total number of cellulose units, attached
exoglucanase enzymes, landing sites and unoccupied sites,
respectively.  Note that each of these sums is finite provided we
specify appropriate initial conditions. We next assume that the constants are independent of $l$ and $i$, so that for each
$l,\ i\in\mathbb{Z}$,
\begin{equation}
\label{locrateeqty}
  \beta^{l,i}=\beta\,, \quad
  \sigma^{l,i}=\sigma\,, \quad
  \gamma_{r}^{ l,i}=\gamma_r \,, \quad
  d_{2A}^{\,l,i}=d_{2A}, \quad
  \alpha^{l,i}=\alpha, \quad
  q^{l,i}=q\,, \quad
  \gamma_{\varrho}^{ l,i}=\gamma_\varrho \,.
\end{equation}
Summing over $l,i$ in \eqref{enz2-NS} and using \eqref{Sloc-NS}, we get
\begin{equation}
\label{totenz2}
\begin{aligned}
  \del_t {e}_{2D}(t)  & = {b}_2 \SP n(t) - \beta S(t) \SP {e}_{2D}(t)
 	+ \sigma \SP {e}_{2A}(t) - d_{2D} \SP {e}_{2D}(t),\\
  \del_t  {e}_{2A}(t) & =   \beta S(t) \SP {e}_{2D}(t) -
	 \big(\sigma + d_{2A} + \gamma_r\big){e}_{2A}(t)\,.
\end{aligned}
\end{equation}

\par

To obtain equations for $\varrho$ and $T$, we scale and add equations
\eqref{Neqn-NS}.  First, we recall
\begin{equation}
\label{cflux}
  \sum_{l,i} i\,r^{l,i}=0 \quad\text{and define}\quad
  r = \sum_{l,i} l\,r^{l,i} = \sum_l l\,r^{l,0}\,,
\end{equation}
so that $r$ represents the total production of cellulose by the environment.
Using \eqref{transcut-NS} and making the change of variable
$j=i-1$ we obtain
\begin{equation*}
\begin{aligned}
\sum_{l,i}i\,&\Big(\widehat{\alpha}^{l,i-1}N^{l,i-1}(t)-\widehat{\alpha}^{l,i} N^{l,i}(t)\Big) \\
&= {\alpha} \sum_{l,i}\Big( i (l-(i-1))N^{l,i-1}(t) - i(l-i)N^{l,i}(t)\Big)
 \\
&={\alpha} \Big( \sum_{l,j}(j+1)(l-j)N^{l,j}(t) - \sum_{l,i}i(l-i)N^{l,i}(t)\Big) \\
& = {\alpha}\Big(  \sum_{l,j} (l-j)N^{l,j}(t)\Big)
 = {\alpha}\SP \big(\varrho(t) - T(t)\big)
\end{aligned}
\end{equation*}
and similarly, using \eqref{translsd-NS} and $j=i+1$,
\begin{equation*}
\begin{aligned}
\sum_{l,i}i\,&\Big(\widehat{\gamma}^{l, i+1} N^{l,i+1}(t) - \widehat{\gamma}^{l,i} N^{l,i}(t)\Big)\\
 & = \gamma_r\sum_{l,i} \Big( i(i+1) N^{l,i+1}(t) - i^2 N^{l,i}(t) \Big)   \\
 & = \gamma_r \Big( \sum_{l,j}(j-1)j N^{l,j}(t) - \sum_{l,i}i^2 N^{l,i}(t) \Big)   \\
& = - \gamma_{r}\sum_{l,j}  j N^{l,j}(t) = -\gamma_r T(t)\,.
\end{aligned}
\end{equation*}

\cmt{ Next, recalling that $e_{2A}^{L+1,i} \equiv 0$ for all $i$, we compute
\begin{equation}\label{calc1}
\begin{aligned}
&	\sum_{l=1}^L \sum_{i=0}^l\,i \Big( {q}^{l+1,i}\SP {e}_{2A}^{l+1,i} + \delta_{li} q^{l+1,i+1}e_{2A}^{l+1,i+1} - {q}^{l,i}\SP {e}_{2A}^{l,i}\Big) \\
&= q\sum_{k=2}^{L} \sum_{i=0}^{k-1}\, i \SP {e}_{2A}^{k,i} - q\sum_{l=1}^L \sum_{i=0}^l i \SP {e}_{2A}^{l,i} + q \sum_{l=1}^{L-1} l e_{2A}^{l+1,l+1}\\
& = - q\sum_{l=1}^{L} \, l \SP {e}_{2A}^{l,l} + q \sum_{l=2}^L (l-1) e_{2A}^{l,l}\\
& = - q \sum_{l=1}^L e_{2A}^{l,l}.
\end{aligned}
\end{equation}
}
% \cmt{Next, using \eqref{transclv-NS}, we compute
% \begin{equation*}
% \sum_{l,i}i\,\Big( {q}^{l+1,i}\SP {e}_{2A}^{l+1,i}(t) - {q}^{l,i}\SP {e}_{2A}^{l,i}(t)\Big)
%  = {q}\sum_{l,i}\Big( i {e}_{2A}^{l+1,i}(t) - i {e}_{2A}^{l,i}(t) \Big) = 0,
% \end{equation*}
% }
and, using \eqref{transdun-NS},
\begin{equation*}
\begin{aligned}
\sum_{l,i}i\,&\Big( \widehat{\theta}^{l,i+1}\SP
	 {e}_{2A}^{l,i+1}(t) - \widehat{\theta}^{l,i}\SP {e}_{2A}^{l,i}(t)\Big)\\
& = {\theta_r (\sigma+d_{2A})}\sum_{l,i}
	\Big( i\,{e}_{2A}^{l,i+1}(t)-i\,{e}_{2A}^{l,i}(t)\Big)\\
& = {\theta_r (\sigma+d_{2A})} 
   \Big(\sum_{l,j} (j-1)\,{e}_{2A}^{l,j}(t)- \sum_{l,i}i\,{e}_{2A}^{l,i}(t)\Big)\\
& = - {\theta_r (\sigma+d_{2A})} \sum_{l,j} {e}_{2A}^{l,j}(t)
= - \theta_r (\sigma+d_{2A}){{e}_{2A}(t)}\,.
\end{aligned}
\end{equation*}

\par

Combining the above identities with \eqref{Neqn-NS} and
\eqref{totqntdef} we conclude
\begin{equation}\label{Ttoteqn}
\begin{aligned}
  \del_t T(t) & = \del_t \Big(\sum_{l,i} i N^{l,i}(t) \Big) \\
  & = {\alpha} \SP
    \big(\varrho(t) - T(t)\big)\, {e}_1(t)
   - \theta_r (\sigma+d_{2A})\, {{e}_{2A}(t)}
   -(\gamma_r+\gamma_\varrho)\, T(t)  \cmt{-q\sum_{l=1}^L e_{2A}^{l,l}}.
\end{aligned}
\end{equation}
Next, referring to \eqref{Sloc-NS}, subtracting
$\del_t {e}_{2A}$ from \eqref{Ttoteqn} and using
\eqref{totenz2}, we obtain
\begin{equation}
\label{Stoteqn}
\begin{aligned}
  \del_t S(t) %& = \bar{r} + \frac{\alpha}{m_1}\SP \bigg(\frac{\varrho(t)}{m_c} - T(t)\bigg)e_1(t)  - \theta_r (\sigma+d_{2A})\frac{e_{2A}(t)}{m_2} -\gamma_r T(t)\\
%  &\quad - \beta S(t) \SP \frac{e_{2D}(t)}{m_2} + \big(\sigma + d_{2A} + \gamma_r\big)\frac{e_{2A}(t)}{m_2}\\
  & = {\alpha} \SP \big(\varrho(t) - S(t)- {{e}_{2A}(t)}\big)\,{e}_1(t) \cmt{-q\sum_{l=1}^L e_{2A}^{l,l}} -  \beta S(t) \SP {{e}_{2D}(t)}\\
  & \qquad{}
  +\Big((1-\theta_r)(\sigma+d_{2A})-\gamma_\varrho\Big) {{e}_{2A}(t)}
  - (\gamma_r+\gamma_\varrho )\, S(t)\,.
\end{aligned}
\end{equation}

\par

We now multiply each term on the right-hand side of \eqref{Neqn-NS} by
$l$ and sum to get an equation for $\varrho(t)$.  First, using
\eqref{transcut-NS} and \eqref{translsd-NS}, we get
\begin{equation*}
\begin{aligned}
\sum_{l,i} l\,\Big(\widehat{\alpha}^{l,i-1}N^{l,i-1}(t)-\widehat{\alpha}^{l,i} N^{l,i}(t)\Big) 
&= {\alpha}\Big( \sum_{i,j}(l-j)N^{l,j}(t) -  \sum_{l,i}(l-i)N^{l,i}(t) \Big)  =0\,,\\[6pt]
 \sum_{l,i} l\,\Big(\widehat{\gamma}^{l,i+1} N^{l,i+1}(t) - \widehat{\gamma}^{l,i} N^{l,i}(t)\Big)  
& = \gamma_r \Big( \sum_{l,j} l\,j\,N^{l,j}(t) - \sum_{l,i} l\,i\, N^{l,i}(t) \Big)=0\,.
\end{aligned}
\end{equation*}
\cmt{
Similarly, we compute
\begin{equation}\label{calc2}
\begin{aligned}
& \sum_{l=1}^L \sum_{i=0}^l l \Big( {q}^{l+1,i}\SP {e}_{2A}^{l+1,i} + \delta_{li} q^{l+1,i+1} e_{2A}^{l+1,i+1}- {q}^{l,i}\SP {e}_{2A}^{l,i}\Big) \\ 
&= q\sum_{l=2}^{L} \sum_{i=0}^{l-1} (l-1) \SP {e}_{2A}^{l,i} - q \sum_{l=1}^L \sum_{i=0}^l l \SP {e}_{2A}^{l,i} + q \sum_{l=1}^{L-1} l e_{2A}^{l+1,l+1}\\
&= q\sum_{{l=2}}^{L} \sum_{i=0}^{l-1} (l-1) \SP {e}_{2A}^{l,i} - q \sum_{l=1}^L \sum_{i=0}^l (l-1) \SP {e}_{2A}^{l,i} - q e_{2A} + q \sum_{l=2}^L (l-1) e_{2A}^{l,l}\\
&= - q e_{2A}.
\end{aligned}
\end{equation}
}

% \cmt{ Similarly, using \eqref{transclv-NS}, we compute
% \begin{equation*}
%  \sum_{l,i}l \Big( {q}^{l+1,i}\SP {e}_{2A}^{l+1,i}(t) - {q}^{l,i}\SP {e}_{2A}^{l,i}(t)\Big) = {q}\Big(\sum_{k,i}(k-1)\, {e}_{2A}^{k,i}(t) - \sum_{l,i}l \,{e}_{2A}^{l,i}(t) \Big) = -q\, {{e}_{2A}(t)},
% \end{equation*}
% }
and, using \eqref{transdun-NS},
\begin{equation*}
\sum_{l,i} l\Big( \widehat{\theta}^{l,i+1}\SP {e}_{2A}^{l,i+1}(t)
	 - \widehat{\theta}^{l,i}\SP {e}_{2A}^{l,i}(t)\Big) =
 {\theta_r  (\sigma+d_{2A})} \sum_{l}l\Big[
  \sum_{i} {e}_{2A}^{l,i+1}(t) - \SP \sum_{i} {e}_{2A}^{l,i}(t)\Big]=0.
\end{equation*}
Combining the above expressions and using
\eqref{totqntdef},\eqref{locrateeqty} and \eqref{cflux},  we obtain
\begin{equation}\label{totrho}
  \del_t \varrho(t)  = \del_t \Big(\sum_{l,i} l N^{l,i}(t)\Big)
  = r - q \, {{e}_{2A}(t)} - \gamma_\varrho\,\varrho(t)\,.
\end{equation}

\par

Finally, by \eqref{transclv-NS}, \eqref{totqntdef} and
\eqref{locrateeqty}, equations \eqref{PNeqn-NS} become
\begin{equation}\label{totpn}
  \begin{aligned}
  \del_t p(t) &= \theta_{p} \, q \, {{e}_{2A}(t)}
       - \gamma\, n(t)\,p(t)-\gamma_p\, p(t)\\[2pt]
  \del_t n(t) &=  \frac{\mu n(t)}{\bar{n}+n(t)}
     \Big( \gamma\,p(t) + (1-\theta_p) \, q\, {
       {e}_{2A}(t)}\Big)
   -\gamma_n \, n(t).
  \end{aligned}
\end{equation}

\par

\noindent{\em $S$-system.}
\cmt{
Note that the equation for $S$ in \eqref{Stoteqn} can be expressed as
\begin{equation}\label{RS}
\left\{\begin{aligned}
&\del_t S(t)  = \alpha \, \bigg(\varrho(t) - S(t)-
{{e}_{2A}(t)}\Big(1+\frac{q \SP e^b_{2A}}{\alpha e_1e_{2A}}\Big)\bigg)\,{{e}_1(t)} -  \beta \, S(t) \, {{e}_{2D}(t)}\\
  & \qquad\qquad
+\Big((1-\theta_r)\,(\sigma+d_{2A})-\gamma_\varrho\Big)\, {{e}_{2A}(t)}
- (\gamma_r+\gamma_\varrho )\, S(t)\,, \quad e^{b}_{2A}:=\sum_{l=1}^L e^{l,l}_{2A}.
\end{aligned}\right.
\end{equation}
When the polymer chains are large, that is $L>>1$, one would expect that the proportion of cellobiose units $lN^{l,l}$ would be small compared to the total number of cellobiose $\varrho = \sum_{l,i} N^{l,i}$ and therefore one can expect $e_{2A}e_1 >> e_{2A}^b $. Then, combining the above equations and dropping the term $\frac{q \, e^b_{2A}}{\alpha e_1 e_{2A}}$ in \eqref{RS} we obtain the $S$-system,
}

Combining the above equations we obtain the
$S$-system,
\begin{equation}\label{Ssyst}
\left\{\begin{aligned}
\del_t \SP {e}_1(t) &= {b}_1 \, n(t) - d_1\, {e}_1(t)\\
\del_t {e}_{2D}(t)  & = {b}_2 \, n(t)  - \beta \,S(t) \, {e}_{2D}(t) +
      \sigma \, {e}_{2A}(t) - d_{2D} \,  {e}_{2D}(t)\\
\del_t  {e}_{2A}(t) & =   \beta S(t) \, {e}_{2D}(t) - \big(\sigma +
      d_{2A} + \gamma_r\big)\,  {e}_{2A}(t)\\
\del_t S(t) & = \alpha \, \Big(\varrho(t) - S(t)-
{{e}_{2A}(t)}\Big)\,{{e}_1(t)} -  \beta \, S(t) \, {{e}_{2D}(t)}\\
  & \qquad
+\Big((1-\theta_r)\,(\sigma+d_{2A})-\gamma_\varrho\Big)\, {{e}_{2A}(t)}
- (\gamma_r+\gamma_\varrho )\, S(t)\\
\del_t \varrho(t) &= r -  q \, {{e}_{2A}(t)}  - \gamma_\varrho\,\varrho(t)\\
  \del_t p(t) &= \theta_{p} \, q \, {{e}_{2A}(t)} - \gamma\,
  n(t)\,p(t)-\gamma_p \,p(t)\\ 
\del_t n(t) &=  \frac{\mu\, n(t)}{\bar{n}+n(t)} \Big( \gamma \,p(t) +
(1-\theta_p)\,  q \, {{e}_{2A}(t)} \Big)-\gamma_n\, n(t).
\end{aligned}\right.
\end{equation}

\begin{remark}\rm

\cmt{Note that system \eqref{Ssyst} is obtained by reduction under the assumption that rates are independent of the length of the chain and the number of landing sites. It is a closed system of seven ODE's which still keeps the important cascading structure of enzymes acting one after another on the cellulose. However, this model assumes that the topology of the cellulose chains (their length and the location of landing sites) does not affect the biodegradation process.}
\end{remark}

\begin{remark}\rm
\cmt{ We make the assumption that the coefficients are constant to derive the $S$-system.  Although this assumption is clearly non-physical, it yields a useful model.  We justify dropping the extra correction term which is
small relative to other terms, in the expectation that the mathematical
error in doing so is much smaller than the modeling errors made by taking
the mean coefficient; this is justified by the numerical experiments in Section \ref{sec:num_exp}.
}
\end{remark}

\cmt{
\subsubsection{$S$-system with fast transitions of $e_{2A}$ and $e_{2D}$}

In this section we will get a modified version of the $S$-system given by \eqref{Ssyst} assuming that the transitions $e_{2D} \to e_{2A}$ and $e_{2A} \to e_{2D}$ are fast.

Recall that $\beta S$ is the rate of transition of $e_{2D}$ to $e_{2A}$ and $\sigma$ is the rate of transition to $e_{2D}$. Thus, if one assumes that $\beta,\sigma \to \infty$ so that $\frac{\beta}{\sigma}$ stays bounded, then the equation \eqref{Ssyst}$_3$ tends to the equilibrium relation
\begin{equation}\label{e2Aeq}
0 = (\beta/\sigma) S(t) \, {e}_{2D}(t) - {e}_{2A}(t) \quad\text{or}\quad
  {e}_{2A}(t) = \omega S(t) \, {e}_{2D}(t) \quad \text{with} \quad \omega := \frac{\beta}{\sigma}.
\end{equation}
Thus, the total number of enzymes $e_{2}$ can be written as
\[
e_{2}=e_{2D}+e_{2A} = e_{2D}+\omega S e_{2D} = e_{2D} (1+\omega S),
\]
so we obtain
\begin{equation}\label{e2depend}
e_{2D} = e_2 R_D(\omega S) \quad\text{and}\quad e_{2A} = e_2 R_A(\omega S) \quad \text{where} \quad R_D(x):=\frac{1}{1+x},\,\,  R_A(x):=\frac{x}{1+x}\,.
\end{equation}
% \[
% \begin{aligned}
% \\ e_{2D} = e_2 R_D(\omega S) \quad \text{with} \quad R_D(x):=\frac{1}{1+x}
% \\ e_{2A} = e_2 R_A(\omega S) \quad \text{with} \quad R_A(x):=\frac{x}{1+x}\,.
% \end{aligned}
% \]
%
Assume next that $d_{2A}=d_{2D}$. Then the total number of enzymes $e_{2}$ satisfies the equation 
\begin{equation}\label{e2eqn}
\del_t e_2 = b_2 n - d_2 e_2 -\gamma_r e_{2A} = b_2 n - ( d_2 + \gamma_r R_A(\omega S)) e_2
\end{equation}
where we have set $d_2=d_{2A}=d_{2D}$. 

Thus, replacing \eqref{Ssyst}$_{2,3}$ with \eqref{e2Aeq} and
\eqref{e2eqn} we obtain the system, called the $S_2$ system,
\begin{equation}\label{S2syst}
\left\{\begin{aligned}
\del_t \SP {e}_1(t) &= {b}_1 \, n(t) - d_1\, {e}_1(t)\\
\del_t {e}_{2}(t)  & = {b}_2 \, n(t)  - (d_2 + \gamma_r R_A(\omega S)) e_2\\
\del_t S(t) & = \alpha \, \Big(\varrho(t) - S(t)-
{{e}_{2}R_A(\omega S)}\Big)\,{{e}_1(t)} -  \beta \, S(t) \, e_2 R_D(\omega S)\\
  & \qquad
+\Big((1-\theta_r)\,(\sigma+d_{2A})-\gamma_\varrho\Big)\, {{e}_{2}R_A(\omega S)}
- (\gamma_r+\gamma_\varrho )\, S(t)\\
\del_t \varrho(t) &= r -  q \, {{e}_{2} R_A(\omega S)}  - \gamma_\varrho\,\varrho(t)\\
\del_t p(t) &= \theta_{p} \, q \, {e_2 R_A(\omega S)} - \gamma\,
  n(t)\,p(t)-\gamma_p \,p(t)\\ 
\del_t n(t) &=  \frac{\mu\, n(t)}{\bar{n}+n(t)} \Big( \gamma \,p(t) +
(1-\theta_p)\,  q \, {{e}_2 R_A(\omega S)} \Big)-\gamma_n\, n(t).
\end{aligned}\right.
\end{equation}
The system \eqref{S2syst} is a version of the $S$-system with fast transitions of $e_{2A}$ and $e_{2D}$. 
}

\subsection{Cleaving Mechanism 2: $T$-model}

\cmt{
We now modify the $S_2$-system derived for the cleaving Mechanism 1 to a model derived for the
cleaving Mechanism 2.  We directly work from the already reduced
$S_2$-system \eqref{S2syst}.  Note
that it is of course possible to derive this model from a more fundamental model
in which we directly implement a corresponding cleaving mechanism on chains of length $l$,
similar to our derivation of the $N$-$S$ model above.

In Mechanism 2, we do not distinguish between attached and detached
enzymes $e_2$.  In addition, we do not distinguish between occupied and
unoccupied sites, preferring to count the total number of available {landing sites} $T(t)$, which satisfies \eqref{TSrel}.  Our goal is to rewrite the $S_2$-system \eqref{Ssyst}  in terms of $T(t)$ rather than $S(t)$. 

Recall that the $S_2$-system \eqref{S2syst} is a version of the $S$-system obtained under the assumption that transitions of the enzyme $e_2$ happen fast, which is equivalent to the requirement \eqref{e2Aeq}. Thus, adding \eqref{Ssyst}$_3$ and \eqref{Ssyst}$_4$ and using the equilibrium relation \eqref{e2Aeq}, we conclude that the total number of landing sites $T=S+e_{2A}$ (with fast $e_2$  transitions) satisfies the equation
\begin{equation}
  \label{Teqne2fast}
  \del_t T(t) =
    \alpha\,\big({\varrho(t)}-T(t)\big)\, {{e}_1(t)}
    -\theta_r\,(\sigma+d_{2})\, {{e}_{2}(t)} R_A(\omega S)
    - (\gamma_r+\gamma_\varrho )\, T(t)\,,
\end{equation}
where we employed the relationship $e_{2A}=e_2 R_A(\omega S)$ and $e_{2D}=e_2R_D(\omega S)$, and the assumption $d_{2}=d_{2D}=d_{2A}$.

We next consider the action of the enzyme $e_2$ in the Mechanism 2. In this model,
the concentration of $e_{2A}$ enzymes is always low as those enzymes are used and degraded immediately after they are produced. The landing, cleaving off of a
cellobiose unit, and detaching all occur at approximately the same instant.  Thus
all enzymes $e_2$ should be considered unattached, so should most
closely resemble $e_{2D}$. To model this scenario we consider the $S_2$-system in the asymptotic regime in which
\begin{equation}\label{reqlimS2T}
\text{$\omega=\frac{\beta}{\sigma} \to 0$ \quad and  \quad $q\omega=q\frac{\beta}\sigma \to \widehat{q}$ }.
\end{equation}
The first assumption makes sure that the enzyme $e_2$ detaches immediately after the attack and as a consequence the proportion between $e_{2A}$ and $e_{2D}$ tends to zero, while the second one is necessary for the population $n(t)$ to survive.

Observe that under assumptions \eqref{reqlimS2T}
\[
\begin{aligned}
R_{2A}(\omega S) \sim \omega S\,, \;\; T = S+e_2R_A(\omega S) \sim S\,, \;\; (\sigma+d_2) e_2 R_A(\omega S) \sim \beta e_2 T\,, \;\; q e_2 R_A(\omega S) \sim q (\beta/\sigma) e_2 T
\end{aligned}
\]
and so the $S_2$-system transforms into the {\em $T$-system} for Mechanism 2\/: 
\begin{equation}\label{Tsyst}
\left\{\begin{aligned}
\del_t \SP {e}_1(t) &= {b}_1 \, n(t) - d_1\, {e}_1(t)\\
\del_t {e}_{2}(t)  & = {b}_2 \, n(t)  - d_2 e_2(t)\\
  \del_t T(t) &=
    \alpha\,\big({\varrho(t)}-T(t)\big)\, {{e}_1(t)}
    -\theta_r\, \beta e_2(t) T(t) 
    - \widehat{\gamma}_{\ssst{T}}\, T(t)\\
\del_t \varrho(t) &= r -  \widehat{q} \,T(t) {e}_{2}(t)  - \gamma_\varrho\,\varrho(t)\\
\del_t p(t) &= \theta_{p} \, \widehat{q} \, T(t) e_2(t) - \gamma\,
  n(t)\,p(t)-\gamma_p \,p(t)\\ 
\del_t n(t) &=  \frac{\mu n(t)}{\bar{n}+n(t)} \Big( \gamma \,p(t) +
(1-\theta_p)\,  \widehat{q} {e}_2(t) T(t)) \Big)-\gamma_n\, n(t).
\end{aligned}\right.
\end{equation}
{where we have set $\widehat{q} = q\,(\beta/\sigma)$ and
$\widehat{\gamma}_{\ssst{T}}=\gamma_r+\gamma_{\varrho}$.}
}

\cmt{Comparing this system to \eqref{Ssyst} modeling Mechanism 1, we make
the following observations: first, the dynamics of $e_2$ and $T$ are
simpler than those of $e_{2A}$, $e_{2D}$ and $S$, because we do not
have to account for the different processes for $e_{2D}$ and $e_{2A}$.
On the other hand, the cleaving rate changes from the linear term
$q\,e_{2A}$ in the $S$-system, to the nonlinear term
$\widehat{q}\,T\,e_2$ which models simultaneously finding and
attacking a landing site.  Note that the coefficient
$\widehat{q} = q\,(\beta/ \sigma)$ combines those coefficients corresponding to
finding sites and cleaving off in Mechanism 1, while
$\widehat{\gamma}_{\ssst{T}}=\gamma_r+\gamma_{\varrho}$ is the combined rate of
degradation of landing sites.}

\subsection{Multiple trait $T$-model}

We now extend the $T$-model to a model that allows for several
species of microorganisms feeding on the same cellulose.
Specifically, we introduce populations $n^i$, with $i\in\{1,...M\}$,
equipped with different traits $x^i$; throughout this section, we use
superscripts to distinguish traits.  We assume that Mechanism 2 is
used by each population to cleave off cellobiose, while the different
populations have different rates for the various actions.

\par

As in the $T$-model, we assume that microorganism $n_i$ produces
endoglucanases enzymes $e_1^i$ that make landing sites, and
exoglucanases enzymes $e_2^i$ that cleave off cellobiose from the
cellulose chains.  In analogy with \eqref{Tsyst}$_{1,2}$, we suppose
that the enzymes $e_1^i$, $e_2^i$ are produced by the microorganism
$n_i$ and degrade with fixed rates. This gives the equations
\begin{equation}\label{ENZM}
\begin{aligned}
\del_t {e}_1^i(t) &= {b}_1^i\, n^i(t) - d_1^i\, {e}_1^i(t)\\
\del_t {e}_2^i(t) &= {b}_1^i\, n^i(t) - d_2^i\, {e}_2^i(t)
\end{aligned}
\end{equation}
where $b_1^i$, $b_2^i$ and $d_1^i$, $d_2^i$ are, respectively, the
enzyme generation and death rates, $i\in \{1,...M\}$.

\par

Next, for simplicity we will not differentiate between landing sites
created by the enzymes of different species.  In other words, the
landing sites made by $e_1^i$ are allowed to be used by any enzyme
$e_2^j$ for all $j$.  Then, following Mechanism $2$, and neglecting
saturation effects, we assume that enzymes $e_1^i$ make landing sites
on the cellulose $\varrho$ with the rate
\begin{equation*}
  \alpha^i\, \big( \varrho (t) - T(t) \big)\,
    {{e}_1^i(t)},
\end{equation*}
where $\alpha^i$ is the probability of an individual enzyme $e_1^i$
finding a spot among the $\varrho(t)-T$ available
cellobiose units (reflecting the requirement that only one landing
site per cellobiose unit is allowed), and making a landing site.
Next, we suppose that the landing sites $T$ are attacked by enzymes
$e_2^i$ with the rate
\begin{equation*}
  \beta^i\, T(t)\, {{e}_2^i(t)}
\end{equation*}
where $\beta^i$ is the probability of an individual $e_2^i$ finding
and attaching to a landing spot.  Finally, as above, we let
$\gamma_{r}>0$ be the decay rate of an individual {landing site}	 and
suppose that the portion $\theta_r^i \in [0,1]$ of those ends is not
usable after an attack by $e_2^i$.  Then, in analogy with
\eqref{Tsyst}$_3$, we obtain the equation for $T$ for multiple trait
populations:
\begin{equation}
\label{TMEQ}
\del_t T(t) = \sum_{j} \alpha^{j} \,
      \big( \varrho (t) - T(t) \big) \,{{e}_1^j(t)}
        - \sum_j \theta_r^j\, \beta^j\, T(t)\, {{e}_2^j(t) }
        - \widehat{\gamma}_{\ssst{T}}\,T(t).
\end{equation}

Next, we let $q^i$ be the number of cellobiose units cleaved off by
$e_2^i$ during an attack.  In analogy with \eqref{Tsyst}$_4$, the
dynamics of cellulose $\varrho$ is then given by
\begin{equation}
\label{RHOM}
  \del_t \varrho(t) = r
    - \sum_j \widehat q^j\,T(t)\,{e}_2^j(t)
    - \gamma_{\varrho} \, \varrho(t)
\end{equation}
{where we set $\widehat q^i=q^i\,\beta^i$, the combined rate of
  attack of $e_2^i$.}

\par

We next let $\theta_p^i\in[0,1]$ denote the proportion of produced
cellobiose produced by $e_2^i$ that is made available to everyone.
Then, as in \eqref{Tsyst}$_5$, the equation for the total amount
$p(t)$ of cellobiose available to everyone is
\begin{equation}\label{PEQM}
\del_t p(t) = \sum_{j} \theta^j_{p} \,
\widehat q^j \, {e}_2^j(t)\, T(t)
   - \sum_j \gamma^j\, n^j(t)\,p(t)-\gamma_p\, p(t)\,.
\end{equation}
where $\gamma^j$ is the predation rate of $p$ by $n^j$, and $\gamma_p$
is the decay rate of $p$.

\par

Finally, we consider the dynamics of the population $n^i$.  First,
recall that cellobiose $p(t)$ is available to all species $n^j$,
$j=1,\dots,M$.  Since every species $n^j$ hunts with the predation
rate $\gamma^j$ on the cellobiose $p$, the growth rate of $n^i$ may be
expressed via the logistic term
\begin{equation}\label{GRP}
  \mu^i\, n^i(t)  \frac{\gamma^i\, p(t)}
          {\bar{n}^i + \frac{1}{\gamma^i}\sum_j \gamma^j\, n^j}
\end{equation}
where $\mu^i$ is the conversion efficiency.

\par

Next, comparing to \eqref{PEQM}, the production rate of cellobiose
which is produced by $e_2^j$ and consumed directly on the spot is
given by
 \begin{equation*}
(1-\theta_p^j) \,{\widehat q^j} \, {e}_2^j(t) \, T(t).
\end{equation*}
We assume that in view of the homogeneity and close proximity of
species, the cellobiose produced by $e_2^j$, $j=1,\dots,M$, can be
consumed by the $n^i$; this manifests the cross species interaction.
We express this as
\begin{equation*}
\mu^i \, (1-\theta_p^j)\, \frac{\nu^{ij} \, n^i}{\nu^{ij} \, \bar{n}^i  +
  \sum_s {\nu^{sj}}n^s} \,\widehat q^j \,
{e}_2^j(t) T(t), \quad \mbox{with} \quad \sum_s{\nu^{sj}} = 1,
\end{equation*}
representing the contribution of the energy obtained from the direct
consumption of cellobiose cleaved off by $e_2^j$ to the growth rate of
$n^i$.  Combining these leads to the equation for the dynamics of the
population $n^i$,
\begin{equation}\label{NEQM}
\begin{aligned}
\del_t n^i(t) &=  \mu^i n^i(t)\left( \frac{\gamma^i p(t)}{\bar{n}^i
    + \frac{1}{\gamma^i}\sum_j \gamma^j\, n^j}+ \sum_j
  \frac{(1-\theta_p^j)\,\nu^{ij} }{\nu^{ij} \,\bar{n}_{i} + \sum_s
    {\nu^{sj}}\,n^s} \,\widehat q^j\,
  {e}_2^j(t)\, T(t)\right)-\gamma_{n}^i\, n^i(t)
\end{aligned}
\end{equation}
where $\gamma_{n}^i$ is the death rate of the population $n^i$.

We collect the above equations to obtain the \emph{multiple trait
  $T$-system},
\begin{equation}\label{TMM-1}
  \left\{\begin{aligned}
    \del_t {e}_1^i(t) &= {b}_1^i\, n^i(t) - d_1^i\, {e}_1^i(t) \\
    \del_t {e}_2^i(t) &= {b}_2^i\, n^i(t) - d_2^i \,{e}_2^i(t) \\
    \del_t T(t) &= \big( \varrho (t) - T(t) \big) \,
          \sum_{j} {\alpha}^{j} \, {e}_1^j(t)  -
          \sum_j \theta_r^j {\beta}^j \, {e}_2^j(t) \,T(t) -
         \widehat{\gamma}_{\ssst{T}}T(t)\\
    \del_t \varrho(t) &= r - \sum_j \widehat{q}^j \, {e}_2^j(t) \,
        T(t) - \gamma_{\varrho} \, \varrho(t)\\
    \del_t p(t) &= \sum_{j}  \theta_{p}^j \,\widehat{q}^j\, {e}_2^j(t) \,
        T(t)- \sum_j \gamma^j\, n^j(t)\,p(t)-\gamma_p\, p(t)\\
    \del_t n^i(t) &=  \mu^i\, n^i(t) \frac{\gamma^i\, p(t)
        }{\bar{n}^i + \frac{1}{\gamma^i}\sum_j \gamma^j\, n^j}\\ &\quad +
        \mu^i n^i(t) \,   \sum_j (1-\theta_p^j)\, \frac{\nu^{ij}
        }{\nu^{ij} \bar{n}^i + \sum_s {\nu^{sj}}\,n^s}\, \widehat{q}^j\,
        {e}_2^j(t)\, T(t) -\gamma_{n}^i\, n^i(t)
\end{aligned}\right.
\end{equation}
where $i=1,\dots,M$.

\par\smallskip

\subsubsection{Compatibility with the single trait model}

We now show that the multiple trait $T$-model directly generalizes the
single trait $T$-model by considering two special cases of the
multiple trait model, and confirming that these reduce to the single
trait model.

Our first test is to assume that all but one species (say the $i$-th)
are absent.  That is, we begin with data $n^j(0) = 0$, and similarly
$e^j_1(0) = e^j_2(0) = 0$, for $j \neq i$.  Then \eqref{TMM-1} implies
that for all $t>0$, $j \neq i$, we have $n^j(t) = 0$.  It is then evident
that  equations \eqref{TMM-1}$_{1,2,3,4,5}$ reduce to
\eqref{Tsyst}$_{1,2,3,4,5}$ (for $n=n^i$, etc), and \eqref{TMM-1}$_6$
becomes
\[
  \del_t n^i(t) = \mu^i\,n^i\,\frac{\gamma^i\, p}{\bar{n}^i + n^i}
  +  \mu^i n^i \,(1-\theta_p^i)\,
  \frac{\widehat{q}^i\,{e}_2^i\,T}{\bar{n}^i + n^i}
  - \gamma_{n}^i\, n^i,
\]
which is exactly \eqref{Tsyst}$_6$.

Our second test is to assume that even though there are $M$ different
traits, the coefficients are independent of $i$, $j$, so there is no
way to distinguish different populations in the model.  In this case,
we check the dynamics for the total population $n(t) = \sum_i n^i(t)$,
and similarly $e_1 = \sum_i e_1^i$ and $e_2 = \sum_i e_2^i$.  It is
then clear that \eqref{TMM-1}$_{3,4,5}$ become
\eqref{Tsyst}$_{3,4,5}$, and adding \eqref{TMM-1}$_{1,2}$ over $i$ gives
\eqref{Tsyst}$_{1,2}$.  Finally, adding \eqref{TMM-1}$_{6}$ over $i$
yields
\begin{align*}
\del_t({\textstyle\sum_i}n^i) &= \mu\,
  {\textstyle\sum_i}n^i\,\frac{\gamma\,p}{\bar n +
  {\textstyle\sum_j}n^j}
   + \mu\,{\textstyle\sum_i}n^i\,
    \frac{1-\theta_p}{\bar n + {\textstyle\sum_s}n^s}\,
    \widehat{q}\,{\textstyle\sum_j} {e}_2^j\,T
  - \gamma_n\,{\textstyle\sum_i}n^i,
\end{align*}
which is again exactly \eqref{Tsyst}$_6$.

\section{Cooperation}
\label{sec:coop}

\subsection{Cooperation in the $T$-model} \label{coopTM}

In this section we consider the system \eqref{Tsyst} on different time
scales. We assume that production of enzymes, consumption and creation
of landing sites occurs at a much faster rate than changes in the
population of the microorganism.  In this case, over time scales on
which the population changes, we can assume that equations
\eqref{Tsyst}$_{1,2,3,4,5}$ are at equilibrium, and the dynamics is
driven by the population change \eqref{Tsyst}$_{6}$.  This gives the
system
\begin{equation}\label{Tsyst-2}
\left\{\begin{aligned}
0 &= {b}_1 \, n(t) - d_1 \,{e}_1(t)	 \\[2pt]
0 & = {b}_2 \, n(t)  - d_2 \,{e}_2(t) \\[2pt]
0 & = \alpha\, \big( {\varrho(t)} - T(t) \big)\, {{e}_1(t)}- \theta_r \,\beta \, T(t) \, {{e}_2(t)}  -\widehat{\gamma}_{\ssst{T}}\, T(t)\\
0 & = r - \widehat{q} \,T(t) \, {e}_2(t)- \gamma_{\varrho} \, \varrho(t)\\
0 & = \theta_{p} \, \widehat{q} \,T(t)\, {e}_2(t) - \gamma\, n(t)\,p(t)-\gamma_p\, p(t)\\[2pt]
\del_t n(t) &=  \frac{\mu\, n(t)}{\bar{n}+n(t)} \Big( \gamma\, p(t) +
(1-\theta_p)\,\widehat{q} \,T(t)\, {e}_2(t) \Big)-\gamma_n\, n(t).
\end{aligned}\right.
\end{equation}

We use \eqref{Tsyst-2} to eliminate all the fast variables, to obtain
a single equation for the population $n$, so that the population
growth rate can be understood.  The first two equations give
\begin{equation}
\label{enzeq}
  {e}_i = {k}_i\, n \quad  \mbox{with} \quad
  \quad {k}_i:=\frac{{b}_i}{d_i}, \quad  i=1,2\,,
\end{equation}
and from \eqref{Tsyst-2}$_4$, we have
\begin{equation*}
  \varrho = \frac{r}{\gamma_{\varrho}} -
     \frac{\widehat{q}}{\gamma_{\varrho}}\, T \, {e}_2 
  = \frac{r}{\gamma_{\varrho}} -
     \frac{\widehat{q}\,{k}_2}{\gamma_{\varrho}}\, T \,n. 
\end{equation*}
Plugging these into \eqref{Tsyst-2}$_3$, we get
\begin{equation*}
  0 = \frac{\alpha\, r\,{k}_1}{\gamma_{\varrho}}\,n - T\;\Big(
     \frac{\alpha \,\widehat{q} \, {k}_2\,{k}_1}{\gamma_{\varrho}}\,n^2
     + ({\alpha\, {k}_1} + \theta_r {\beta \,{k}_2})\,n
     + \widehat\gamma_r \Big),
\end{equation*}
so that
\begin{equation}
\label{Teq}
  T = \frac{\alpha\, r\,{k}_1}{\gamma_{\varrho}}\,\frac{n}{P_2(n)},
\end{equation}
where $P_2(n)$ is the quadratic polynomial
\begin{equation}
  \label{P2n}
  \begin{aligned}
    P_2(n) & = c_2\,n^2 + c_1\,n + c_0, \\
    c_2 = \frac{\alpha\,\widehat{q}\,{k}_2\,{k}_1}{\gamma_{\varrho}},\quad
    &c_1 = {\alpha\, {k}_1} + {\theta_r\,\beta \,{k}_2},
    \quad   c_0 = \widehat\gamma_r.
  \end{aligned}
\end{equation}
Next, using these in \eqref{Tsyst}$_5$, it follows that
\begin{equation}
  \label{peq}
  p = \theta_p\,\widehat{q}\, {k}_2\,\frac{T\,n}{\gamma\,n+\gamma_p}
    = \theta_p\,\widehat{q} \, {k}_2\,\frac{\alpha\, r\,{k}_1}
	{\gamma_{\varrho}}
	\;\frac{n^2}{(\gamma\,n+\gamma_p)\,P_2(n)}.
\end{equation}

% \begin{equation}\label{neqbeq}
%  \Big(\frac{q_c m_c \beta}{m_2 \gamma} \Big) \Big(\frac{1}{n+\frac{\gamma_p}{\gamma}}\Big)T e_2
%    = \theta_p  \Big(\frac{n^2}{n+\frac{\gamma_p}{\gamma}}\Big) \bar{c} \Big[ c_2 n^2  +c_1 n + (\gamma_r+\gamma_{\varrho}) \Big]^{-1}\,,\quad \bar{c}=\Big(\frac{\alpha \SP r k_1 q_c \beta}{m_1 m_2 \gamma \gamma_{\varrho}} \Big)\,.
% \end{equation}

Finally, we use \eqref{Teq}, \eqref{peq} in \eqref{Tsyst-2}$_6$ to get
the scalar population equation
\begin{equation}
\label{neqbeq}
  \del_t n(t) = n(t) \Big[ B(n) - \gamma_n \Big],
\end{equation}
where
\begin{equation}\label{Bdef}
\begin{aligned}
B(n) &=  \frac{\mu}{\bar{n}+n} \bigg( \gamma \,p
         + (1-\theta_p)\,\widehat{q} \, {k}_2\,T\,n\bigg) \\
     &= \mu\,\widehat{q} \, {k}_2\,\frac{T\,n}{\bar{n}+n}\,\bigg(
     \frac{\gamma\,\theta_p}{\gamma\,n+\gamma_p} + (1-\theta_p) \bigg)
     \\   &= K\,n^2\,\Phi(n),
% & = \Big(\frac{\mu q_c m_c \beta k_2}{m_2}\Big) \frac{n}{n+\bar{n}}\Big[ \theta_p \frac{1}{n+\bar{\gamma}}+(1-\theta_p) \Big]T\\
% & = {c_0}\frac{n^2}{n+\bar{n}}\Big[ \theta_p \frac{1}{n+\bar{\gamma}}+(1-\theta_p) \Big] \big[ c_2 n^2 + c_1 n  + (\gamma_{r}+\gamma_{\varrho})  \big]^{-1}\,,\quad {c_0}=  \Big(\frac{\mu q_c \beta \alpha \SP r k_1 k_2 }{m_1 m_2 \gamma_{\varrho} }\Big)\,.
\end{aligned}
\end{equation}
%\begin{figure}[t] \centering
%\includegraphics*[width=4.5in, height=3.5 in]{Bpic.eps}
%\caption{the graph of $B=B(n)$} \label{Bpic}
%\end{figure}
where the function $\Phi(n)$ and constant $K$ are given by
\begin{equation}\label{PhiK}
\begin{aligned}
  \Phi(n) &= \bigg(\frac{\theta_p}{n+\gamma_p/\gamma} +
     1-\theta_p\bigg)\,\frac{1}{(n+\bar{n})\,P_2(n)}\,, \\[5pt]
  K &= \mu\,\widehat{q} \, {k}_2\,\frac{\alpha\, r\,{k}_1}
	{\gamma_{\varrho}}
    = \frac{\mu\,q\, \beta\,\alpha\,r\,{b}_1\,{b}_2}
           {\gamma_\varrho\,d_1\,d_2}\,.
\end{aligned}
\end{equation}

\subsection{Asymptotics of $B(n)$}

We are interested in the structure of $B(n)$ for
small populations, $n \ll \bar n$.  First, we note that the birth rate
$B(n)$ is positive and satisfies
\[
  B(0) = 0  \quad\text{and}\quad  \lim_{n\to\infty} B(n) = 0,
\]
so that $B$ is globally bounded.

For $n$ small, using \eqref{Bdef}, \eqref{PhiK} and \eqref{P2n}, we have
\begin{equation}
  \label{B/n2}
  \frac{B(n)}{n^2}= K\,\Phi(n) \approx K\,\Phi(0) =
  \frac{K}{\bar{n}\,\widehat\gamma_{r}}\,
  \Big(\theta_p\,\frac{\gamma}{\gamma_p}+1-\theta_p \Big),
\end{equation}
so that $B(n) \sim n^2$ for $n \ll \bar n$.  Thus for small
populations, $B(n)$ is convex, and so superlinear.  This superlinear
birth rate is indicative of cooperative behavior.

More generally, the growth rate $B(n)/n$ increases as long as
$\log\big(n\,\Phi(n)\big)$ does, and
\begin{equation*}
  \frac{\del}{\del n}\Big(\frac{B(n)}{n}\Big)
  = \frac{\del}{\del n}\Big(K\,n\,\Phi(n)\big)
  = K\,n\,\Phi(n)\,\Big(\frac1n + \frac{\del_n\Phi}{\Phi}\Big)
  = K\,n\,\Phi(n)\,\Big(\frac{\del}{\del n}\log(n\,\Phi)\Big),
\end{equation*}
so the system exhibits cooperative behavior as long as
\[
  \frac{\del}{\del n}\log(n\,\Phi) = \frac 1n
  - \frac{\theta_p}{(n+\gamma_p/\gamma)^2}\,
     \Big(\theta_p\,\frac{1}{n+\gamma_p/\gamma} + 1-\theta_p\Big)^{-1}
   - \frac1{n+\bar n} - \frac{2c_2n +c_1}{P_2(n)} > 0\,.
\]
This condition is at least true for $n\in(0,n_*)$, where $n_*$ is the
smallest positive root of this expression; combining the fractions, it
is evident that $n_*$ is the smallest positive root of a fifth-order
polynomial.

Moreover, referring to \eqref{B/n2}, we see that
\begin{equation*}
  \frac{\del}{\del {\theta_p}}K\,\Phi(n)\Big|_{n=0} =
  \frac{K}{\bar{n}\,\widehat\gamma_{r}}\,
  \Big( \frac{\gamma}{\gamma_p}-1\Big),
\end{equation*}
which is positive if and only if $\gamma > \gamma_p$.  For small
population $n$, we expect this to persist: that is,
\[
  \frac{\del}{\del \theta_p}\Big(\frac{B(n)}{n^2}\Big) =
  \frac{\del}{\del {\theta_p}}K\,\Phi(n) > 0 \quad\text{if and only
    if} \quad \gamma > \gamma_p.
\]
Since $\theta_p\in[0,1]$ is the proportion of produced cellulose which
is shared, this last inequality suggests that for small populations
{\it sharing food is beneficial} in terms of growth as long as the
consumption rate $\gamma$ is greater than the decay rate $\gamma_p$ of
the cleaved off cellobiose.

% \par\medskip

% \noindent{\bf Properties of $B(n)$.} For any choice of positive parameters \eqref{Tsyst-2} and $\theta_p \in [0,1]$, one can show that $B(n)$ satisfies: % (see Fig. \ref{Bpic}):
% \begin{itemize}
% \item [(1)] $B(0)=0$, $B(n)>0$ for $n>0$, and $B(\infty)=0$
%  \item[(2)] There exists unique $n^*>0$ such that $B(n)<B(n^*)=:B^*$ for all $n \neq n^*$.
% \item [(3)] Moreover, $B\uparrow$ for $n \in(0,n^*)$ and $B \downarrow$ for $n\in(n^*,\infty)$\,.
% \end{itemize}

% \par\medskip

% \noindent{\bf Equilibrium states.} Note that $e_1,e_2,T,p,\varrho$ are uniquely expressed in terms of $n$. Thus, the equilibrium states of \eqref{neqbeq} will automatically determine equilibrium values of other variables. To this end we consider the equation $n(B(n)-\gamma_n) = 0$ and conclude that:
% \begin{itemize}
% \item [(1)] If $\gamma_n>B^*$ then the equation \eqref{neqbeq} has no equilibrium solutions, except $n\equiv 0$.

%  \item[(2)] If $\gamma_n=B_0$ then there are two equilibrium solutions $n \equiv 0$ and $n \equiv n^*$.

%      \item[(3)] If $\gamma_n<B_0$ then there exist three equilibrium states: $n_0$, $n_1$, and $n_2$ such that $$n_0=0<n_1<n^*<n_2.$$
% \end{itemize}

\subsection{Cooperation in multiple-trait $T$-model}

As in Section \ref{coopTM}, we consider the system \eqref{TMM-1} on a
generational time scale.  We again assume that production of enzymes,
consumption and creation of landing sites happens at a much faster
rate then change in the populations $n^i$.  In other words, we assume
that the equations \eqref{TMM-1}$_{1,2,3,4,5}$ are at equilibrium, and
the dynamics of the system is driven by the population equations
\eqref{TMM-1}$_{6}$. This results in the system
\begin{equation}\label{TMMPD}
  \left\{\begin{aligned}
      0 &= {b}_1^i\, n^i(t) - d_1^i\,{e}_1^i(t), \\
      0 &= {b}_2^i\, n^i(t) - d_2^i \,{e}_2^i(t), \\
      0 &= \big( {\varrho (t)} - T(t) \big) \,
      \sum_{j} {\alpha}^{j} \, {e}_1^j(t)  -
      \sum_j \theta_r^j\, {\beta}^j \, {e}_2^j(t) \,T(t) -
      \widehat{\gamma}_{\ssst{T}}\,T(t),\\
      0 &= r - \sum_j \widehat{q}^j \, {e}_2^j(t) \,
      T(t) - \gamma_{\varrho} \, \varrho(t),\\
      0 &= \sum_{j}  \theta_{p}^j \,\widehat{q}^j\, {e}_2^j(t) \,
      T(t)- \sum_j \gamma^j\, n^j(t)\,p(t)-\gamma_p\, p(t),\\
      \del_t n^i(t) &=  \mu^i\, n^i(t)\, \frac{\gamma^i p(t) \,
      }{\bar{n}^i + \frac{1}{\gamma^i}\sum_j \gamma^j\, n^j},\\ &\quad +
      \mu^i\, n^i(t) \,   \sum_j (1-\theta_p^j)\, \frac{\nu^{ij}
      }{\nu^{ij}\, \bar{n}^i + \sum_s {\nu^{sj}\,}n^s} \widehat{q}^j \,
      {e}_2^j(t)\, T(t) -\gamma_{n}^i\, n^i(t).\\
\end{aligned}\right.
\end{equation}

\par

We wish to understand the growth rate of $n^i$ as a function of
$n=(n_1,\dots,n_M) \in \RR^M$.  The first two equations of
\eqref{TMMPD} give
\begin{equation}\label{EEQMM}
   {e}_1^i(t) = {{k}_1^i} n^i(t), \quad
   {e}_2^i = {{k}_2^i} n^i, \quad  \mbox{with} \quad
   k^i:=\frac{b^i}{d^i}, \quad  i=1,\dots,M\,.
\end{equation}
We next write $n$ as a vector and introduce the coefficient
vectors
\begin{equation}
  \label{coeffvec}
\begin{aligned}
  A &= \Big( {\alpha}^j \, {k}_1^j\Big), \quad
  B = \Big(\theta_r^j\, {\beta}^j\, {k}_2^j\Big), \quad
  Q = \Big(\widehat{q}^j\,{k}_2^j\Big), \\
  \Theta &= \Big(\theta_p^j\, \widehat{q}^j\,{k}_2^j\Big), \quad
  \Gamma = \Big(\gamma^j\Big), \quad
  N^k = \Big(\nu^{jk}\Big),
 \end{aligned}
\end{equation}
and denote the scalar product by $\langle\cdot,\cdot\rangle$.  We can
then rewrite \eqref{TMMPD}$_{3,4,5}$ as
\begin{align*}
  0 &= \big( \varrho (t) - T(t) \big)\, \langle A,n\rangle  - \langle B,n\rangle\,T(t) -
  \widehat{\gamma}_{\ssst{T}}\,T(t), \\
  0 &= r - \langle Q,n\rangle\,T(t) - \gamma_{\varrho} \, \varrho(t),\\[2pt]
  0 &= \langle\Theta,n\rangle\,T(t) - \langle\Gamma,n\rangle\,p(t)-\gamma_p\, p(t).
\end{align*}
These immediately yield
\begin{equation}
\label{varrhop}
  \varrho(t) = \frac{1}{\gamma_{\varrho}} \,\big( r - T(t)\, \langle
       Q,n\rangle \big) 
\quad\text{and}\quad
  p(t) = \frac{\langle\Theta,n\rangle}{\langle\Gamma,n\rangle + \gamma_p}\,T(t),
\end{equation}
and, plugging in, we get
\begin{equation}\label{Tsol}
  0 = \frac{r}{\gamma_{\varrho}}\,\langle A,n\rangle - \tau(n)\,T(t),
\quad\text{so that}\quad
  T(t) = \frac{r}{\gamma_{\varrho}}\,\frac{\langle A,n\rangle}{\tau(n)},
\end{equation}
where we have set
\begin{equation}
  \label{taudef}
  \tau(n) = \frac{1}{\gamma_{\varrho}}\,\langle A,n\rangle\,\langle Q,n\rangle
   + \langle A,n\rangle + \langle B,n\rangle + \widehat{\gamma}_{\ssst{T}}\,.
\end{equation}

Finally, using \eqref{varrhop} and \eqref{Tsol} in
\eqref{TMMPD}$_6$, we can write our population system as
\begin{equation}
\label{NEQM-2}
  \del_t n^i  = n^i \, \big( B^i(n) -\gamma_n^i \big),
\end{equation}
where the $i$-th population's birth rate is
\begin{equation}
\label{Bi(n)}
\begin{aligned}
  B^i(n) &= \frac{\mu^i\,\gamma^i\,p(t)}
            {\bar n^i + \frac{1}{\gamma^i}\,\langle\Gamma,n\rangle}
  + \mu^i\,\sum_j \frac{(1-\theta_p^j)\,\nu^{ij}\,\widehat{q}^j\, {k}_2^j\,n^j}
                       {\nu^{ij}\,\bar n^i+\langle N^j,n\rangle}\,T(t)\\
  &= \mu^i\,T(t)\,\bigg(\frac{\gamma^i\,\langle\Theta,n\rangle}
    {(\bar n^i + \frac{1}{\gamma^i}\,\langle\Gamma,n\rangle)\,(\langle\Gamma,n\rangle + \gamma_p)}
  + \sum_j \frac{(1-\theta_p^j)\,\nu^{ij}\,\widehat{q}^j\, {k}_2^j\,n^j}
                       {\nu^{ij}\,\bar n^i+\langle N^j,n\rangle}\bigg)\\
  &= \frac{\mu^i\,r}{\gamma_{\varrho}}\,\frac{\langle A,n\rangle}{\tau(n)}\,
    \bigg(\frac{\gamma^i\,\langle\Theta,n\rangle}
    {(\bar n^i + \frac{1}{\gamma^i}\,\langle\Gamma,n\rangle)\,(\langle\Gamma,n\rangle + \gamma_p)}
  + \sum_j \frac{(1-\theta_p^j)\,\nu^{ij}\,\widehat{q}^j\, {k}_2^j\,n^j}
            {\nu^{ij}\,\bar n^i+\langle N^j,n\rangle}\bigg)\,.
\end{aligned}
\end{equation}
Here the two terms in the growth rate correspond to intentionally shared
food and food consumed as it's produced, respectively.

\par\medskip

\noindent{\bf Asymptotic behavior of $B^i(n)$.}\quad
Assuming the coefficients are positive, we make the following
observations about the birth rate $B^i(n)$.  According to
\eqref{taudef} $\tau(n)$ is quadratic in $n$, while all inner
products in \eqref{NEQM-2} are linear.  It follows immediately that
$B^i(n)\to 0$, and in fact $B^i(n) = O(\frac1n)$ as $n\to\infty$.

We are more interested in the behavior for small populations,
$n\sim0$.  Since $\tau(0) = \widehat{\gamma}_{\ssst{T}}$, no denominators
vanish, and \eqref{Bi(n)} yields
\[
  B^i(n) = O\big(({\textstyle\sum}n)^2\big)  \quad\text{for}\quad
  n \sim 0.
\]
More precisely, recalling that
\[
  \nabla_n\langle V,n\rangle = V \quad\text{and}\quad
  D^2_n\big(\langle V,n\rangle\,\langle W,n\rangle\big)
  = V\,W^T + W\,V^T,
\]
we see that at $n=0$, the gradient of $B^i$ vanishes,
$\nabla_nB^i(0)=0$, and the Hessian of $B^i$ is the symmetric matrix
\[
  D^2_nB^i(0) = \frac{\mu^i\,r}{\gamma_{\varrho}\,
    \widehat{\gamma}_{\ssst{T}}\,\bar n^i}\,\bigg(
    \frac{\gamma^i}{\gamma_p} (A\Theta^T + \Theta A^T)
    + A(Q-\Theta)^T + (Q-\Theta)A^T \bigg).
\]
We cannot conclude that $B^i$ is convex as the matrix $D^2_nB^i(0)$
is not positive definite, but because all the entries are positive, we
can conclude that the directional derivative is increasing in any
direction in the positive orthant $\{ n^k \ge 0 \}$, which indicates
cooperative behavior.

\par\medskip

\noindent{\bf Special Case.}\quad
Now, we consider the special case when $\nu^{ij}=0$ for $i \neq j$; in
this case, there is no competition for cellobiose that is not
intentionally shared.  In this special case, \eqref{Bi(n)} becomes
\begin{align*}
  B^i(n) &= \frac{\mu^i\,r}{\gamma_{\varrho}}\,
    \frac{\langle A,n\rangle}{\tau(n)}\,
    \bigg(\frac{\gamma^i\,\langle\Theta,n\rangle}
    {(\bar n^i + \frac{1}
{\gamma^i}\,\langle\Gamma,n\rangle)\,(\langle\Gamma,n\rangle + \gamma_p)}
  + \frac{(1-\theta_p^i)\,\widehat{q}^i\, {k}_2^i\;n^i}
            {\bar n^i + n^i}\bigg)\\
  &=: B^i_1(n) + B^i_2(n).
\end{align*}

As in the single-trait case, we again see an indication that for small
populations, more sharing (represented by the coefficient vector
$\Theta$) would be beneficial for the $i$-th population provided
$\gamma^i>\gamma_p$, because it increases the derivative
$\nabla_nB^i(n)$: this can be seen by differentiating with respect to
the vector parameter $\Theta$.  Recall that $\gamma^i>\gamma_p$ means
that cellobiose is consumed (by $n^i$) faster than it decays.

\begin{lemma}\label{coopmult}
Suppose that $\nu^{ij}=0$ for $i \neq j$. Let $i\in \{1,\dots,M\}$ be fixed and let
\[
  n_0 = (n_0^1,n_0^2,\dots,n_0^{i-1},0,n_0^{i+1},\dots, n_0^M)
  \in\RR^M \quad \mbox{with} \quad n_0^j\ge
 0.
\]
Then
\begin{equation}\label{DB2isc}
\frac{\del B_2^i }{\del n^i}(n_0) =
  \frac{\mu^i\,r}{\gamma_{\varrho}}\,
    \frac{\langle A,n_0\rangle}{\tau(n_0)}\,
    \frac{(1-\theta_p^i)\,\widehat{q}^i\,{k}_2^i}{\bar n^i} > 0.
\end{equation}
Furthermore, suppose $\min_j\bar{\alpha}^j>0$ and $\min_j\gamma^j >
0$. Then there exists $\eps >0$ such that for all $\max_j\theta_p^j
<\eps$, we have
\[
\frac{\del B^i}{\del n^i}(n_0) > 0 \quad \mbox{for all $n_0\ne 0$}\,.
\]
\end{lemma}

\par

\begin{proof}[Idea of proof]
Equation \eqref{DB2isc} follows immediately by differentiation.  When
we differentiate $B^i_1$, we introduce negative terms each time the
derivative falls on a denominator.  However, each such term introduces
a higher power in the denominator, so each of those terms can be
represented as a product of $\langle A,n_0\rangle/\tau(n_0)$ with
terms uniformly bounded in $n_0$.  Comparing these to \eqref{DB2isc},
it follows that by choosing $\max_j\theta_p^j <\eps$ with $\eps$ small
enough, the sum will be positive.
\end{proof}

\cmt{
\begin{remark}\rm
The key conclusion in Lemma \ref{coopmult} is that the birth rate $B^i(n)$ includes some form of {\em cooperation}. Namely, the cooperative effect is not eliminated  in the multiple trait population $T$-model provided that $\nu_{ij}=0$ for $i \neq j$. In other words, when interspecies competition for `ready to be digested' resource  $p$  is not involved then there is an Allee effect for each species. Compare it for instance with simple logistic terms like $r-n_0^i$ which decreases with $n_0^i$. In contrast $B^i(n)$ actually penalizes populations which are too small (and populations which are too large of course just like a logistic term). Finally, even if the conditions of Lemma 4.1 do not hold a cooperative effect still maybe present depending on the differential $D B(0)$. 
\end{remark}
}

%%%%%%%%%%%%%%%%%%%%%%%%%%%%%%%%%%%%%%%%%%%%%%%%%%%%%%%%%
\subsection{Cooperative interactions in $N$-$S$-model} \label{coopNS}
%%%%%%%%%%%%%%%%%%%%%%%%%%%%%%%%%%%%%%%%%%%%%%%%%%%%%%%%%

We now consider cooperation in the $N$-$S$-model as we did for the
simpler models.  We are interested in the situation that the
production of enzymes, consumption and creation of landing sites
occurs much faster than changes in the population of the
microorganism.  Thus we again assume that equations \eqref{enz1-NS},
\eqref{enz2-NS}, \eqref{Neqn-NS} and \eqref{PNeqn-NS}$_1$ are at
equilibrium, and the dynamics is driven by the population change
\eqref{PNeqn-NS}$_2$.  We also assume that the length of the cellulose
chains does not exceed a given number $L>0$.  Then we obtain the
system
\begin{equation}\label{NS-fast}
\left\{\begin{aligned}
 0  &= {b}_1 \SP n(t) - d_1 {e}_1(t)\\[2pt]
 0  & = {b}_2 \SP n(t)  - \sum_{l,i} \bigg[\beta^{l,i}  \Big( i N^{l,i}(t) - {{e}_{2A}^{l,i}(t)}\Big) {e}_{2D}(t) - \sigma^{l,i} \SP {e}_{2A}^{l,i}(t)\bigg] - d_{2D} \SP {e}_{2D}(t)\\
0  & =    \beta^{l,i}\Big( i N^{l,i}(t) - {{e}^{l,i}_{2A}(t)}\Big) {e}_{2D}(t)  - \Big(\sigma^{l,i} + d_{2A}^{\,l,i} + \gamma_{r}^{ l,i}\Big) {e}_{2A}^{l,i}(t)\\[3pt]
0  &= r^{l,i}  +
\Big(\widehat{\alpha}^{l,i-1}N^{l,i-1}(t)-\widehat{\alpha}^{l,i}
N^{l,i}(t)\Big){e}_1(t)
 + \Big(\widehat{\gamma}^{l,i+1} N^{l,i+1}(t) - \widehat{\gamma}^{l,i}
 N^{l,i}(t)\Big) \\ & \quad
+ \Big( {q}^{l+1,i}\SP {e}_{2A}^{l+1,i}(t) -
{q}^{l,i}\SP {e}_{2A}^{l,i}(t)\Big)
+\cmt{\delta_{li}q^{l+1,l+1}e_{2A}^{l+1,l+1}}\\
&\quad +\Big( \widehat{\theta}^{l,i+1}\SP {e}_{2A}^{l,i+1}(t) -
\widehat{\theta}^{l,i}\SP {e}_{2A}^{l,i}(t)\Big)
 - \gamma_{\varrho}^{l,i}N^{l,i}(t)
\\
0 &= \theta_{p} \, \sum_{l,i} {q}^{l,i} \SP {e}_{2A}^{l,i}(t) - \gamma n(t)p(t)-\gamma_p p(t)\\
\del_t n(t) &=  \frac{\mu n(t)}{\bar{n}+n(t)} \Big( \gamma  +
\frac{1-\theta_p}{\theta_p}(\gamma n(t) + \gamma_p) \Big) p(t)
  -\gamma_n n(t)\,,
\end{aligned}\right.
\end{equation}
for $(l,i) \in I_L$, where $\delta_{li}$ is the Kronecker delta, the
rates $r^{l,i}=0$ when $i \neq 0$ and we use the convention
\eqref{convind}; here we have used \eqref{NS-fast}$_5$ to simplify
\eqref{NS-fast}$_6$.

\par

To describe the dynamics, it is sufficient to express $p$ in terms of
$n$, which will in turn provide a scalar autonomous differential
equation for $n(t)$.  For small populations the
equations \eqref{NS-fast}$_{1-5}$ can be solved uniquely in terms of
$n$, yielding the following theorem.

\begin{theorem}\label{eqsolthm}
  There are $m,\bar{m}>0$ and $C^{\infty}$ functions
\begin{equation}
\label{NSsoln}
\widehat{e}_1(n), \quad \widehat{e}_{2D}(n), \quad  \widehat{e}_{2A}^{l,i}(n), \quad \widehat{N}^{l,i}(n), \quad \widehat{p}(n) : (-m,\bar{m}) \to \RR
\end{equation}
such that:
\begin{itemize}
\item [$(i)$] For each $n \in (-m,\bar{m})$ the equations
  \eqref{NS-fast}$_{1-5}$ can be solved uniquely for $e_1$, $e_{2D}$,
  $e_{2A}^{l,i}$, $N$, $p$ in terms of $n$,
\begin{equation*}
  e_1=\widehat{e}_1(n), \quad e_{2D}=\widehat{e}_{2D}(n), \quad
  e_{2A}^{l,i}=\widehat{e}_{2A}^{l,i}(n), \quad
  N^{l,i}=\widehat{N}^{l,i}(n), \quad  p=\widehat{p}(n).
\end{equation*}

\item [$(ii)$] The functions from \eqref{NSsoln} are given to leading
  order as
\begin{equation}
\label{soln1}
  \widehat{N}^{l,i}(n) = \nu^{l,i} \SP n^i + O(n^{i+1})\,,
\end{equation}
with
\begin{equation*}
  \nu^{l,0}=\frac{r^{l,0}}{\gamma_{\varrho}^{l,0}}\quad \mbox{and} \quad \nu^{l,i}= \nu^{l,i-1} \SP \frac{\widehat{\alpha}^{l,i-1}}{\widehat{\gamma}^{l,i}+\gamma_{\varrho}^{l,i}} \frac{{b}_1}{d_1} \,, \quad i \geq 1\,,
\end{equation*}
and with

\begin{align}
\widehat{e}_1(n) &= \frac{{b}_1}{d_1}n, \qquad \widehat{e}_{2D}(n) =
\frac{{b}_2}{d_{2D}}n+O(n^2)  \nonumber\\\nonumber
\widehat{e}_{2A}^{l,i}(n) &=
i\bigg(\frac{  {b}_2  \SP \beta^{l,i} \SP \nu^{l,i} \,
}{d_{2D}(\sigma^{l,i} + d_{2A}^{\,l,i} + \gamma_{r}^{ l,i})}
+O(n)\bigg)n^{i+1} \\[5pt]
\frac{\gamma_p}{\theta_p} \SP \widehat{p}(n)
&= \bar{p}(n) \SP n^2, \quad \mbox{where} \quad
\bar{p}(n)=\frac{{b}_2}{d_{2D}} \sum_{l} q^{l,1} \frac{  \beta^{l,1} \SP
  \nu^{l,1} \,   }{(\sigma^{l,1} + d_{2A}^{\,l,1} + \gamma_{r}^{
    l,1})}  + O(n)\,.
\label{solnp}
\end{align}
\end{itemize}
\end{theorem}

% \begin{proof}[Sketch of proof]
% The quantities $e_1$, $e_{2D}$, $e_{2A}^{l,i}$ and $p$ are uniquely
% determined in terms of $n$ and $N^{l,i}$ by the relations
% \eqref{NS-fast}$_{1,2,3}$.  Substituting these relationships into
% \eqref{NS-fast}$_4$ and employing the implicit mapping theorem allows
% one to conclude that $N^{l,i}$ is uniquely determined by
%  $n\in(-m,\bar{m})$ for some $m,\bar{m}>0$, and by induction, we
% discover the formulas for $\widehat{N}^{l,i}(n)$ given by
% \eqref{soln1}.  This in turn allows us to obtain the quantities
% $e_1$, $e_{2D}$, $e_{2A}^{l,i}$ and $p$ in terms of $n$ on the interval
% $(-m,\bar{m})$.
%\end{proof}
\cmt{
\begin{proof}%[RY analysis of \eqref{NS-fast}]
\def\l#1{\eqref{NS-fast}$_#1$}
For fixed $t$, we regard \l1-\l5 as an algebraic system, which can be
solved uniformly in $t$, for small $n$.  First, \l1 yields
\[
  e_1 = k_1\,n,
\quad\text{with}\quad
  k_1 = \frac{d_1}{b_1},
\]
and subtracting all of \l3 from \l2 and solving gives
\begin{equation}\label{e2D}
  e_{2D} = \frac{b_2}{d_{2D}}\,n -
  \sum_{l,i}\frac{d_{2A}^{\,l,i} + \gamma_{r}^{ l,i}}{d_{2D}}
        \,e_{2A}^{l,i}.
\end{equation}
Next, \l5 yields
\[
  p = \frac{\theta_{p}}{\gamma n + \gamma_p}\,
       \sum_{l,i} {q}^{l,i} \, {e}_{2A}^{l,i},
\]
while \l3 gives
\begin{equation}\label{e2A}
  {e}_{2A}^{l,i} = i\, N^{l,i}\,
\bigg(1 + \frac{\sigma^{l,i} + d_{2A}^{\,l,i} + \gamma_{r}^{ l,i}}
              {\beta^{l,i}\,{e}_{2D}}\bigg)^{-1}
= \eta^{l,i}\,N^{l,i},
\end{equation}
where we have set
\[
\begin{aligned}
  \eta^{l,i} &= i\,\Big(1 - \frac 1{1+\zeta^{l,i}\,{e}_{2D}}\Big)
    = i\,\zeta^{l,i}\,{e}_{2D}\,\big(1+O(e_{2D})\big),
\quad\text{with}\\
  \zeta^{l,i} &= \frac{\beta^{l,i}}
             {\sigma^{l,i} + d_{2A}^{\,l,i} + \gamma_{r}^{ l,i}}.
\end{aligned}
\]
Now, regarding $n$ and $e_{2D}$ as fixed, we use \eqref{e2A} in \l4 to
get 
\begin{equation}\label{Nsys}
\begin{aligned}
  0  &= r^{l,i}  +
\Big(\widehat{\alpha}^{l,i-1}N^{l,i-1}-\widehat{\alpha}^{l,i}
N^{l,i}\Big)\,k_1\,n
 + \Big(\widehat{\gamma}^{l,i+1} N^{l,i+1} - \widehat{\gamma}^{l,i}
 N^{l,i}\Big) \\ & \quad
+ \Big({q}^{l+1,i}\,\eta^{l+1,i}\,N^{l+1,i} -
{q}^{l,i}\,\eta^{l,i}\,N^{l,i}\Big)
+\cmt{\delta_{l,i}\,q^{l+1,l+1}\,\eta^{l+1,l+1}\,N^{l+1,l+1}}\\
 & \quad
+\Big( \widehat{\theta}^{l,i+1}\,\eta^{l,i+1}\,N^{l,i+1} -
\widehat{\theta}^{l,i}\,\eta^{l+1,i}\,N^{l,i}\Big)
- \gamma_{\varrho}^{l,i}N^{l,i},
\end{aligned}
\end{equation}
which we regard as a linear system,
\[
   A\,\mathbf{N} = \mathbf{r}, \quad\text{for}\quad
  \mathbf{N}^T = \big(N^{1,0},\ N^{1,1},\ N^{2,0},\
         N^{2,1},\ N^{2,2},\ N^{3,0},\dots N^{L,L} \big).
\]
When expressed in matrix form, the matrix $A$ is sparse and upper
Hessenberg, with subdiagonal entries $-\widehat{\alpha}^{l,i-1}k_1n$.
It follows that the matrix is upper triangular for $n=0$, so
invertible for small $n$, and we get a unique solution
$N^{l,i}=N^{l,i}(n,e_{2D})$.

Setting $n=e_{2D}=0$, and recalling that $r^{l,i}=0$ for $i>0$ and
$\widehat\gamma^{l,0}=0$, \eqref{Nsys} gives
the initial solution
\[
  N^{l,0}(0,0) = \frac{r^{l,0}}{\gamma_\varrho^{l,0}}, \qquad
  N^{l,i}(0,0) = 0, \ i>0,
\]
and so using \eqref{e2A} and \eqref{e2D}, we get in particular
$e_{2A}^{l,i} = \eta^{l,i} = 0$ whenever $n=e_{2D}=0$.

We now plug the solution $N^{l,i}=N^{l,i}(n,e_{2D})$ into \eqref{e2A},
\eqref{e2D} to get
\[
  G(n,e_{2D}) := e_{2D} - \frac{b_2}{d_{2D}}\,n +
  \sum_{l,i}\frac{d_{2A}^{\,l,i} + \gamma_{r}^{ l,i}}{d_{2D}}
   \,\eta^{l,i}\,N^{l,i}(n,e_{2D})
  = 0,
\]
relating $e_{2D}$ to $n$.  We calculate
\[
  \frac{\del G}{\del e_{2D}}\Big|_{(0,0)} = 1,
\]
so the implicit function theorem implies that, for $n$ small enough,
there is a unique function $e_{2D}(n)$ such that $G(n,e_{2D}(n))=0$,
and moreover,
\[
  e_{2D} = \frac{b_2}{d_{2D}}\,n + O(n^2).
\]

Finally, we have
\[
  N^{l,i}(n) = N^{l,i}(0,0) + O(n) \quad\text{and}\quad
  \eta^{l,i} = i\,\frac{b_2\,\zeta^{l,i}}{d_{2D}}\,n + O(n^2),
\]
so according to \eqref{e2A}, we have 
\[
  e_{2A}^{l,i} = O(n^2), \quad\text{so also}\quad
  p = O(n^2).
\]
Moreover, for small $n$, we can write \eqref{Nsys} for $i\ge 1$ as
\[
\begin{aligned}
  0  =& \widehat{\alpha}^{l,i-1}\,k_1\,n\,N^{l,i-1} -
   \big(\widehat{\gamma}^{l,i} + \gamma_{\varrho}^{l,i}
    + O(n)\big)\,N^{l,i}\\
  &\quad + \big(\widehat{\gamma}^{l,i+1} +O(n)\big)\,N^{l,i+1}
  +\cmt{\delta_{l,i}\,q^{l+1,l+1}\,\eta^{l+1,l+1}\,N^{l+1,l+1}},
\end{aligned}
\]
and we can solve this inductively in $i$, to get
\[
  N^{l,i} = \frac{\widehat{\alpha}^{l,i-1}\,k_1}
         {\widehat{\gamma}^{l,i} + \gamma_{\varrho}^{l,i}}\,N^{l,i-1}\,
         n\,\big(1+O(n)\big),
\]
from which the result follows.
\end{proof}}

\medskip

\noindent {\bf Birth rate $B[n]$.}\quad
By Theorem \ref{eqsolthm}, and using \eqref{NS-fast}$_6$ and
\eqref{solnp}, we conclude that for small populations
$n\in[0,\bar{m})$, the dynamics is again driven by the equation
\begin{equation*}
\del_t n(t) =  n(t)\big( B(n)-\gamma_n \big),
%= n(t) \bigg[\frac{\mu }{\bar{n}+n(t)} \Big( \gamma  +
%\Big(\frac{1}{\theta_p}-1\Big) (\gamma n + \gamma_p) \Big) \widehat{p}(n)
%  -\gamma_n \bigg]\\
\end{equation*}
where now the birth rate $B(n)$ is given by
\begin{equation}\label{NSbirth}
\begin{aligned}
B(n) &= \frac{\mu}{\bar{n}+n(t)} \Big( \gamma  +
  \frac{\theta_p-1}{\theta_p} (\gamma n + \gamma_p) \Big) \widehat{p}(n)\\[6pt]
 &= \frac{\mu\,n^2}{\bar{n}+n(t)} \Big( \frac{\gamma}{\gamma_p}\theta_p  + \big(1-\theta_p\big) (\frac{\gamma}{\gamma_p} n + 1) \Big) \bar{p}(n)\,.
\end{aligned}
\end{equation}

\par

We now divide $B(n)$ by $n^2$ and differentiate with respect
to $\theta_p$.  Since $\bar{p}$ is independent of $\theta_p$, we
obtain
\[
  \frac{\del }{\del \theta_p}\bigg(\frac{B(n)}{n^2}\bigg)\bigg|_{n=0} =
  \frac{\mu}{\bar{n}}\, \Big( \frac{\gamma}{\gamma_p} - 1 \Big)\,
  \frac{{b}_2}{d_{2D}} \SP m_c\sum_{l} q^{l,1} \frac{  \beta^{l,1} \SP \nu^{l,1}
    \,   }{(\sigma^{l,1} + d_{2A}^{\,l,1} + \gamma_{r}^{ l,1})},
\]
so that
\[
  \frac{\del }{\del \theta_p}\bigg(\frac{B(n)}{n^2}\bigg)\bigg|_{n=0}
   >0   \quad \mbox{if and only if} \quad   \gamma>\gamma_p\,.
\]
Thus we arrive at a similar conclusion to that of the $T$-model: that
is, for small populations, sharing food (within the species) is
beneficial in terms of growth as long as the consumption rate $\gamma$
is greater than the decay rate $\gamma_p$ of the cleaved off
cellobiose.

\section{A Model in the Continuous Setting}
\label{sec:conttrait}

{In this section, by analogy with our multiple trait $T$-model
\eqref{NEQM-2}, we develop a model for a population with any number of traits, which we write for convenience in the continuous setting.  We should say again that in our context, we do not expect the number of possible enzymes for cellulose degradation to be that large but note that this model includes
finite trait models by appropriate use of $\delta$-functions,
\[
  n(t,x)=\sum_j n_j(t)\,\delta_{x_j}(x).
\]
We hence present the continuous model here for its generality and because the resulting
equation may be more amenable to analysis.}

\par

We think of the multiple-trait population as having $M$
traits indexed by $x_1 < \dots < x_M$, so we write
\[
  n^i(t)=n(x^i,t) \Delta x, \quad
  e_1^i(t)=e_1(x^i,t) , \quad
  e_2^i(t)=e_2(x^i,t),
\]
for some functions $n(x,t)$, $e_1(x,t)$, $e_2(x,t)$, representing the
population and enzyme densities.  We now simply assume that the
variable $x$ takes on a continuous range of values.

We similarly translate the coefficient vectors \eqref{coeffvec}, so
that these become continuous parameters: that is, we allow the
parameters $b_i$, $d_i$, $\alpha$, $\beta$, $\widehat{q}$, $\theta_p$,
etc, to depend on $x$, and in analogy to \eqref{coeffvec}$_{1,2,3}$ we
set
%\begin{equation}
%  \label{ctscoeff}
%\begin{gathered}
%  A(x) = \widehat{\alpha}(x) \,k_1(x), \quad
%  B(x) = \theta_r(x)\, \widehat{\beta}(x)\, k_2 (x) , \quad
%  Q(x) = \widehat{q}(x)\,k_2(x), \\[3pt]
%%  \Theta(x) = \theta_p(x)\, \widehat{q}(x)\,k_2(x), \quad\text{and}\quad
%%  \Gamma(x) = \gamma(x),
% \end{gathered}
%\end{equation}
\begin{equation}
  \label{ctscoeff}
\begin{gathered}
  A(x) = \alpha(x) \,\frac{{b}_1(x)}{d_1(x)}, \quad
  B(x) = \theta_r(x)\, {\beta}(x)\, \frac{{b}_2(x)}{d_2(x)} , \quad
  Q(x) = \widehat{q}(x)\,\frac{{b}_2(x)}{d_2(x)}, \\[3pt]
%  \Theta(x) = \theta_p(x)\, \widehat{q}(x)\,k_2(x), \quad\text{and}\quad
%  \Gamma(x) = \gamma(x),
 \end{gathered}
\end{equation}
where these are now positive functions.  In particular, we interpret $\alpha$, $ \beta$ and $\widehat{q}$ as the rates of landing site generation, occupation, and the rate of cellobiose production, per individual, respectively.

We now simply follow the development that led to \eqref{NEQM-2}, but
reinterpreting the inner product, so that for each function $W(x)$,
\[
  \langle W,n\rangle = \int W(y)\,n(y,t)\;dy.
\]
Then \eqref{varrhop}, \eqref{Tsol} and \eqref{taudef} are unchanged.
To express the population equation, we define the functional
\[
  \tau[n] := \frac{1}{\gamma_{\varrho}m_c}\langle A,n\rangle \langle
  Q,n\rangle  + \langle A,n\rangle +\langle B,n\rangle  + \gamma_p,
\]
and the convolution
\[
  N[n](x,t) = \int \nu(s,x)\,n(s,t)\;ds,
\]
which is an inner product in the first variable.

We must now model the last term in \eqref{Bi(n)}.  In analogy with
that term, we define
\[
  \Xi(z;x,t) = \frac{\nu(x,z) Q(z)}
     {\nu(x,z)\,\bar n(x,t) +  N[n](z,t)}\,.
\]

Now, in analogy with \eqref{NEQM-2}, \eqref{Bi(n)}, we write the
population equation as
\[
  \del_t n(x,t)  = n(x,t) \, \Big( B[n](x,t) -\gamma_n(x) \Big),
\]
where the birth rate is now the functional
%\[
%  B[n] = \frac{\mu(x)\,r}{\gamma_{\varrho}\, m_c}\,
%   \frac{\langle A,n\rangle}{\tau[n]}\,
%    \bigg(\frac{\Gamma(x)\,\langle\Theta,n\rangle}
%    {(\bar n(x) + \frac{1}{\gamma(x)}\,\langle\Gamma,n\rangle)
%  \,(\langle\Gamma,n\rangle + \gamma_p)}
%  + \Big\langle\Xi(\cdot;x,t),n\Big\rangle\bigg)\,.
%\]

\begin{equation}\label{Bcont}
  B[n] = \frac{\mu(x)\,r}{\gamma_{\varrho}\, m_c}\,
   \frac{\langle A,n\rangle}{\tau[n]}\,
    \bigg( \frac{\langle \theta_p \, Q, n\rangle \SP \gamma(x)}  {\big(\bar n(x) + \frac{1}{\gamma(x)}\,\langle \gamma ,n \rangle \big)
  \,\big(\langle \gamma ,n\rangle + \gamma_p\big)}
+ \Big\langle (1-\theta_p) \, \Xi(\cdot \SP ;x,t) \,, n\Big\rangle \bigg)\,.
\end{equation}

\section{Numerical experiments}\label{sec:num_exp}
% Here's an experiment that would further justify the assumption: generate
% coefficients which depend on $l$ and $i$ via some reasonable distributions
% (different for different coefficients), and see if the corresponding
% dynamics is qualitatively similar to that of the S-system, which uses mean
% values for the coefficients.  If so, that would be a good numerical
% justification for simplifying the model.

\cmt{

In this section we use numerical experiments to test and compare the various models that we have derived.

\par\smallskip

\noindent{\it $N$-$S$ and $S$ models.} First, we compare the $N$-$S$-model with the $S$-model. The results of the numerical computations are presented in Figure \ref{comp1}. The coefficients of the $S$-system are chosen as 
\begin{equation}
\begin{aligned}
& b_1 = 0.5,& &b_2 = 0.5,& &d_1 =  0.5,& &d_{2A} = 0.5,& &d_{2D}=0.5&\\
&\beta = 0.5,& &\sigma = 0.1,& &\gamma_r   = 0.01,& &\gamma_p=0.001,& &\gamma_{\varrho} = 0.001&\\
&\gamma_n = 0.1,& &\alpha = 0.05,& &r = 1000,& &\theta_r = 0.05,& &\theta_p = 0.75&\\
& q = 1,& &\gamma = 0.005,& &\mu = 0.5,& &\bar{n} = 100.& &&
\end{aligned}
\end{equation}
The coefficients of the $N$-$S$-system are chosen randomly as follows. For each coefficient of the $S$-system, let us call it `c', the corresponding coefficient $c^{l,i}$ of the $N$-$S$-system (which explicitly depends on the state of the chain $(l,i)$) is chosen as 
\begin{equation}
c^{l,i}=c X \quad \text{where} \quad X\sim\text{gamma}(k,\theta), \quad k=p^{-2}, \theta=p^2\,.
\end{equation}
Here the value $p$ is the standard deviation of $X$. Thus all samples approximately lie  in $(c-3p,c+3p)$.

\begin{figure}[H]

\centering

% first group of figures
\begin{minipage}{0.31\textwidth}
\includegraphics[width=\linewidth]{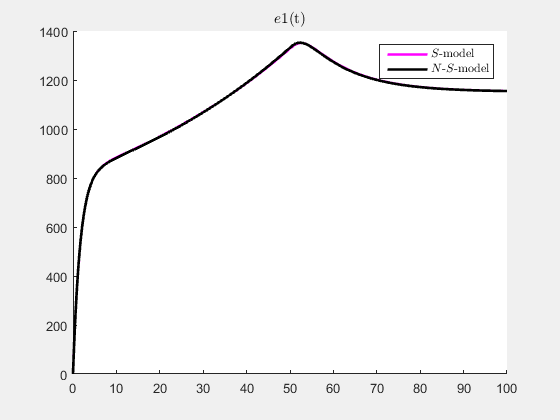}
\end{minipage}
\hspace{3mm} % choose horizontal spacing to suit your needs
\begin{minipage}{0.31\textwidth}
\includegraphics[width=\linewidth]{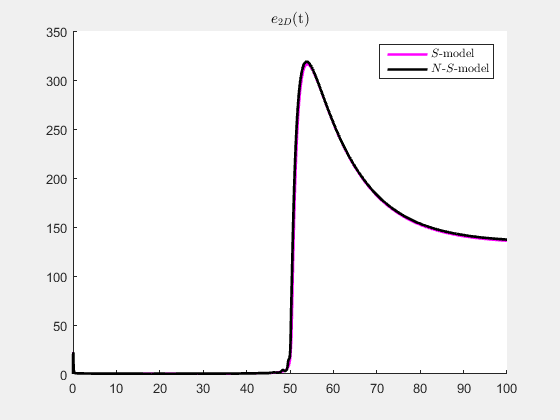}
\end{minipage}

% second group of figures
\smallskip
\begin{minipage}{0.31\textwidth}
\includegraphics[width=\linewidth]{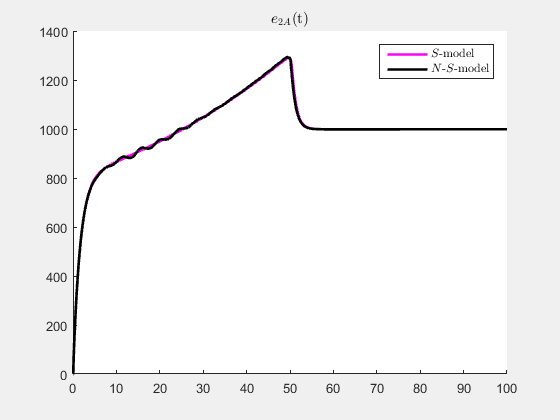}
\end{minipage}
\hspace{3mm} % choose horizontal spacing to suit your needs
\begin{minipage}{0.31\textwidth}
\includegraphics[width=\linewidth]{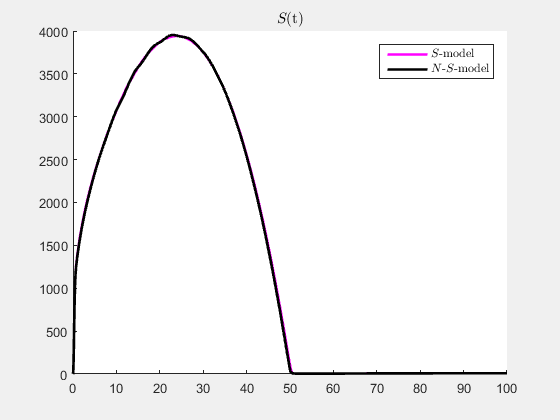}
\end{minipage}

% second group of figures
\smallskip
\begin{minipage}{0.31\textwidth}
\includegraphics[width=\linewidth]{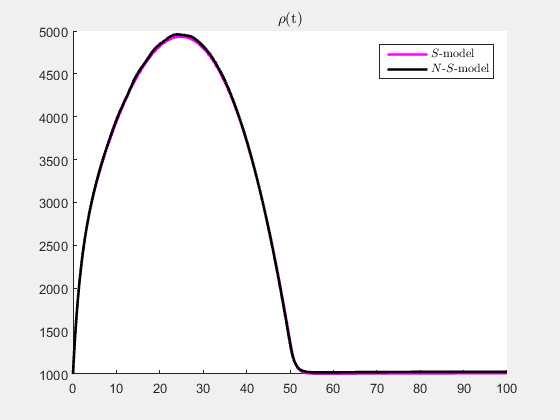}
\end{minipage}
\hspace*{\fill}
\begin{minipage}{0.31\textwidth}
\includegraphics[width=\linewidth]{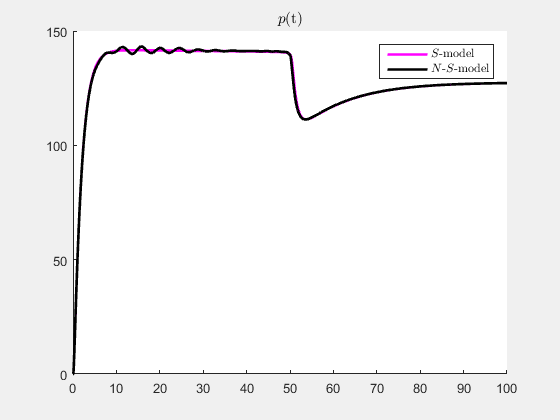}
\end{minipage}
\hspace*{\fill}
\begin{minipage}{0.31\textwidth}
\includegraphics[width=\linewidth]{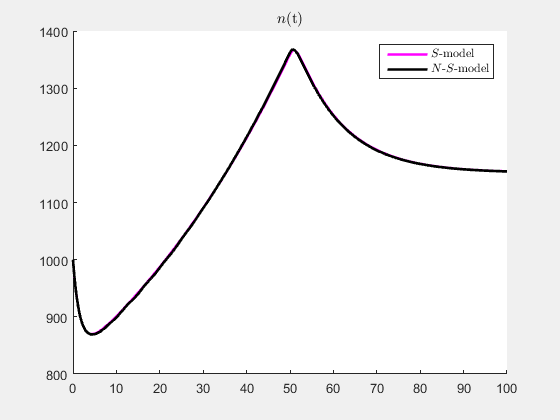}
\end{minipage}
\caption{Solutions of S and NS systems with $3p=0.05$}\label{comp1}
\end{figure}
\par\smallskip

\noindent{\it $S$ and $T$ models.}
We next compare the $T$ and $S$ systems numerically. The coefficients of the $S$-system are chosen as 
\begin{equation}
\begin{aligned}
&b_1=0.1,& &b_2 =  0.1,& &d_1 =  0.1,& &d_{2A} = 0.1,& &d_{2D}=0.1&\\ 
&\gamma_r=0.01,& &\gamma_p=0.005,& &\gamma_{\varrho} = 0.005,& &\gamma_n = 0.01,& &\alpha = 0.05&\\
&r = 1000,& &\theta_r = 0.05,& &\theta_p = 0.75,& &\gamma = 0.01,&  &\mu = 0.5,&\\
&\bar{n} = 100&
\end{aligned}
\end{equation}
and 
\begin{equation}
\begin{aligned}
& \{\beta\}_{i=1}^6 =  \{0.6668, 0.8394, 1.0567,  1.3304, 1.6748, 2.1085\}&\\
& \{\sigma\}_{i=1}^6 = \{17.7828, 177.828, 1778.28, 17782.8, 177828, 1778280\}\\ 
& \{q\}_{i=1}^6 = \{0.0007, 0.0053, 0.0421, 0.3342, 2.6544, 21.0848\}. 
\end{aligned}
\end{equation}
The pictures on Figure \ref{comp2} correspond to the limiting procedure where
\[
q_i \frac{\beta_i}{\sigma_i} = q_0=0.000025 \quad\text{and} \quad \frac{\beta_i}{\sigma_i} \to 0.
\]
}

% 0.118568685283083
% 0.037494710466623
% 0.011856868528308
% 0.003749471046662
% 0.001185686852831
% 0.000374947104666

\begin{figure}[H]
\begin{subfigure}[b]{0.32\textwidth}
  \begin{center}
    \includegraphics[width=1.0\textwidth, height=\textwidth]{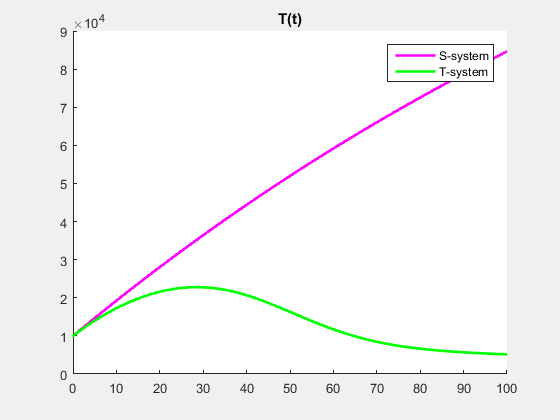}
  \end{center}
   \caption{$i=1, \frac{\beta_i}{\sigma_i}=0.11857$}
\end{subfigure}
\hfill
\begin{subfigure}[b]{0.32\textwidth}
  \begin{center}
    \includegraphics[width=1.0\textwidth, height=\textwidth]{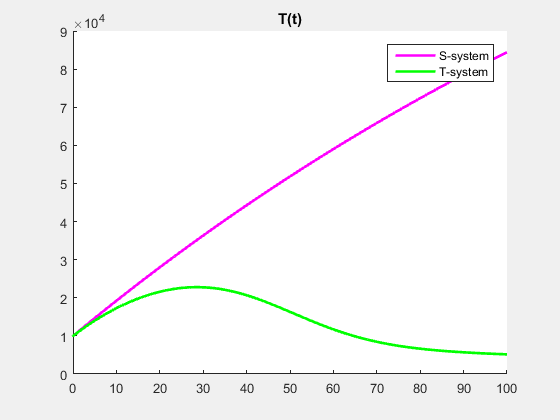}
  \end{center}
  \caption{$i=2, \frac{\beta_i}{\sigma_i}=0.03749$}
\end{subfigure}
\hfill \vspace{5pt}  
\begin{subfigure}[b]{0.32\textwidth}
  \begin{center}
    \includegraphics[width=1.0\textwidth, height=\textwidth]{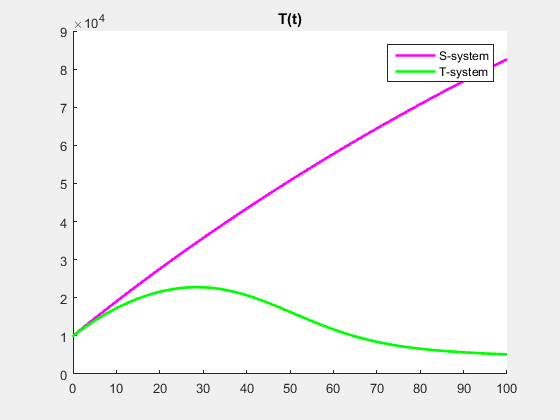}
  \end{center} 
\caption{$i=3,\frac{\beta_i}{\sigma_i}=0.01185$}
\end{subfigure}
\hfill \vspace{5pt}  
\begin{subfigure}[b]{0.32\textwidth}
  \begin{center}
    \includegraphics[width=1.0\textwidth, height=\textwidth]{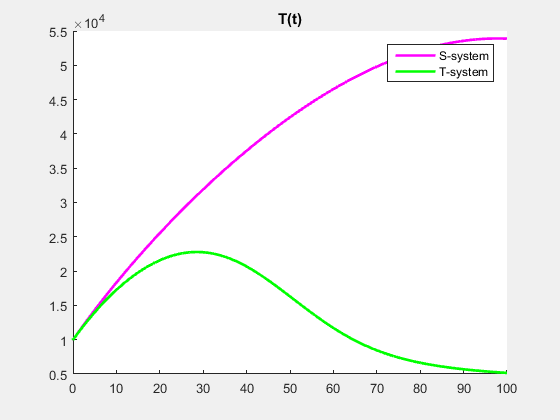}
  \end{center} 
\caption{$i=4,\frac{\beta_i}{\sigma_i}=0.00374$}
\end{subfigure}
\hfill \vspace{5pt}  
\begin{subfigure}[b]{0.32\textwidth}
  \begin{center}
    \includegraphics[width=1.0\textwidth, height=\textwidth]{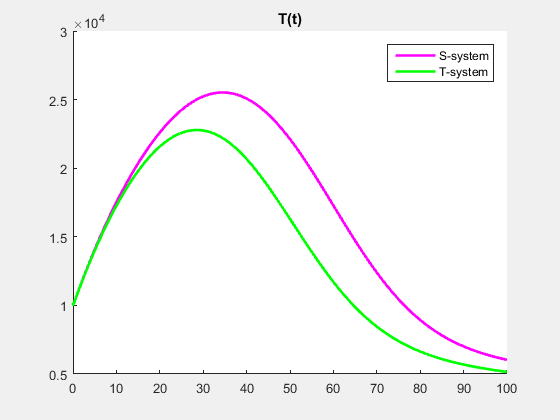}
\end{center}
\caption{$i=5,\frac{\beta_i}{\sigma_i}=0.00118$}
\end{subfigure} \vspace{5pt}  
\hfill
\begin{subfigure}[b]{0.32\textwidth}
  \begin{center}
  \includegraphics[width=1.0\textwidth, height=\textwidth]{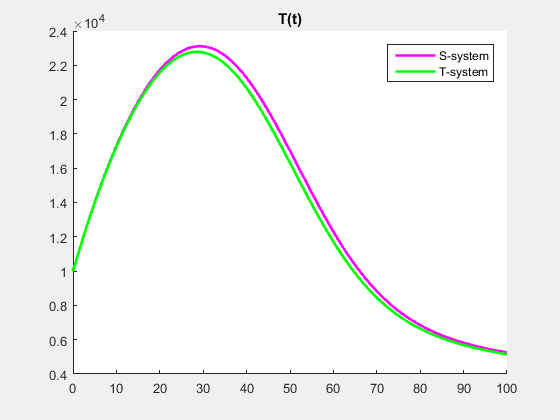}
  \end{center} 
\caption{$i=6,\frac{\beta_i}{\sigma_i}=0.00037$}
\end{subfigure}
\caption{$T(t)$ of $T$-model and $S$-model for $q_i \frac{\beta_i}{\sigma_i} =q_0=0.000025$ and $\frac{\beta_i}{\sigma_i} \to 0$.}\label{comp2}
\end{figure}

\section{Appendix}

\subsection{Tail issue in deterministic selection dynamics}  

Models in population dynamics focus on selection because it is rightfully
viewed as the main mechanism to explain the survival of populations and the evolution of traits.  The selection mechanism in these models is often driven by competition between individuals, possibly combined with mutations to create new traits.  In addition competition is well understood from the modeling point of view.

On the other hand cooperative effects are harder to model, especially
at the level of micro-organisms. Several well-known cooperative
effects (such as sexual reproduction for large animals) do not take
place for all micro-organisms.  Nevertheless, the importance of such
effects has long been recognized: see for instance the works
\cite{Ferriere02, Jones09,Jones12,Holland10} on mutualism that discuss
interspecies interactions yielding reciprocal benefits.

\par

In this paper we introduce biological mechanisms, by the example of cellulose bio-degradation, that lead to reproduction rates encoding both (intra-species) cooperative effects and competition between individuals; see Section \ref{sec:coop}.  This suggests that reproduction rates that only incorporate competition may fail to describe many biochemical processes, especially at the level when $B[n]$ significantly deviates from traditional logistic terms, that is for small populations.

\par

There are several approaches to study the phenotypical evolution driven
by small mutations in replication, the main objective being to
describe the dynamics of the fittest (or dominant) trait in the
population.  The main mechanisms affecting dynamics are usually a) the
selection principle (due to competition, birth and death), and b)
small mutations.  These two mechanisms influence the trait dynamics on
two different scales.  The selection effect becomes evident on the
reproduction timescale $t_R$, while the effect of small mutations is
evident on a generation timescale $t_M \gg t_R$.  The drastic
difference between the two scales introduces both small and large
parameters into models (mutations can be small or rare for instance,
population is usually large and death rates could vary) and this
causes various difficulties.

One of the best known approaches is the so-called adaptive dynamics
theory, see for instance \cite{Dieckmann96, ODieckmann04, Calcina04,
  Ghosal03, Metz96}.  Adaptive dynamics considers evolution as a series
of invasions by a small mutant population of the dominant trait
population, a process which is classically modeled by a system of
ODE's.  Depending on the relative fitness of the mutant, this can lead
to the replacement of the dominant trait or the extinction of the
invading population (the cases of co-existence are usually harder to
handle at this level).

Other very popular models are stochastic, or individual-centered
models, see for instance
\cite{Ferriere02,Champagnat06-1,Champagnat06-1} among many.
Probabilistic models are natural because they take natural
fluctuations of births, deaths, and mutations into account at the
individual level, and are therefore considered to be the most
realistic.  They consist in life and death processes for each
individual $X_i$.  A typical example consists in taking Poisson
processes with birth rates $b(X_i)$ and death rates increasing with
the competition between individuals, for example $d_i = d(X_i)+\sum_{j
  \neq i} I(X_i-X_j).$

When a birth occurs, it simply adds another individual with the same
trait, unless a mutation takes place, generally with small
probability.  In that case, the new individual has a different random
trait, obtained through some distribution.  In general of course
competition could influence both the birth rate and the mortality
rate.  Under the right scalings, stochastic models can lead to the
classical adaptive dynamics \cite{Meleard09, Perthame09}.

When the total number of individuals is too large (it can easily reach
$10^{10}-10^{12}$ for some micro-organisms), stochastic models can
become cumbersome and prohibitive to compute numerically for
instance.  In that case, one expects to be able to derive a
deterministic model as a limit of large populations which would be
simpler to use.  Such a derivation was provided in \cite{Champagnat08}
for example, leading to integro-differential equations such as
\begin{equation}
\label{ctraitmut}
  \del_t n(x,t) = \Big( r(x)-\int I(x-y)n(t,y) \, dy \Big)n(x,t) +M[n](t,x),
\end{equation}
where $r(x)=b(x)-d(x)$ and $M[\cdot]$ is the mutation kernel, a
diffusion or integral operator.  This is the level of modeling that we
are interested in this article.

Even though deterministic models of type \eqref{ctraitmut} are
obtained from stochastic ones, simulations for these two types of
models typically produce different behaviors in terms of evolutionary
speeds and branching patterns.  In stochastic simulations, in which a
single individual represents a minimal unit necessary for survival,
demographic stochasticity (the variability in population growth rates
among individuals) acts drastically on small populations, leading to
complete extinction of small populations with negative reproduction
rates.  In deterministic models however, sub-populations can never go
completely extinct and can ``rebound'' later on if their reproduction
rate becomes positive.

It is an open and difficult question of how to keep the stochastic
effects for the small populations in the deterministic models.
Perthame and Gauduchon \cite{Perthame09} made an attempt in truncating
the populations with less than one individual by introducing an
analog of stochastic mortality for models of type \eqref{ctraitmut}, a
survival threshold, which allows phenotypical traits in the small
population to vanish in finite time.  In \cite{Perthame09} this is
achieved by modifying \eqref{ctraitmut} as follows
\begin{equation}\label{ctraitmut2}
  \del_t n(x,t)  = \Big( r(x)- I \star n \Big) n
           -\sqrt{\frac{n}{\bar{n}}}+M[n](t,x)\,.
\end{equation}
The new term enables the population to vanish for some traits when the
population density is too low in comparison with $\bar{n}$, which
disallows densities corresponding to fewer than one individual.

\par

As one wishes to see the evolution of traits generated by
mutations, one needs to rescale the above equation in time.  This leads
to large deviation type phenomena which can be observed by defining
$n_{\eps}(t,x)=\exp(\phi_\eps(t,x)/\eps)$, with $\eps$ the ratio of the
reproduction and mutation time scales (see \cite{Perthame09,
  ODieckmann05}).
%leads to the equation
%\begin{equation}\label{ctraitmut3}
%  \del_t n_{\eps}(x,t)  = \frac{1}{\eps}\Big( r(x)-I \star n_{\eps}
%  \Big) n_{\eps}
%  -\frac{1}{\eps}\sqrt{\frac{n_{\eps}}{\bar{n}_{\eps}}}+M_{\eps}[n](t,x)\,.
%\end{equation}
One now has two scales for the populations, the small population
threshold $\bar{n}$ and the exponential scale $\exp \phi/\eps$.

Often, the aim is to analyze the population behavior in the limit as
$\eps \to 0$ and therefore $\bar{n}$ should be chosen in terms of
$\eps$.  Numerical simulations for the corresponding equation with
initial data of monomorphic type, see \cite{Perthame09}, indicate that
the evolution speeds and time of branching depend on this choice of
$\bar{n}$ in terms of $\eps$. When $\eps$ is fixed, too large a value
of $\bar{n}$ leads to extinction, while too small a value of $\bar{n}$
leads to spontaneous jumps in branching, see \cite{Perthame09,
  ODieckmann05}.
% Thus, choosing for $\bar{n}_{\eps}=\bar{n}$ a positive constant
% independently of $\eps$ leads to extinction as $\eps$ becomes
% small. Therefore, when one accelerates time, the mortality threshold
% $\bar{n}_{\eps}$ must be rescaled.

\par

A complete mathematical analysis of the general equations is currently
intractable.  One of the few situations that is currently understood
\cite{Perthame09} is when the mortality threshold is chosen as
$\bar{n}_{\eps}=\exp({-\frac{\bar{\varphi}}{\eps}})$.
% For a given scaling Perthame and Gauduchon \cite{Perthame09} analyse
% the limiting behaviour of populations by studying the limits of
% phase functions $\varphi_{\eps}(x,t)=\eps \log n_{\eps}(x,t)$ that
% solve Hamilton-Jacobi equations of the form
%\begin{equation*}
%  \del_t \varphi_{\eps} = r- I \star n_{\eps}
%  -e^{\bar{\varphi}-\varphi_{\eps}}+ H(x,\nabla \varphi_{\eps})\,.
%\end{equation*}
% Using the theory of viscosity solutions to Hamilton-Jacobi equations,
% the authors pass to the limit and recover the equation for the weak
% limit $\varphi$ of $\varphi_{\eps}$.
%\par
However, the scaling
$\bar{n}_{\eps}=\exp(-\frac{\bar{\varphi}}{\eps})$ for a fixed
$\bar\varphi$ is often much too small.  Recall that the threshold
$\bar{n}_{\eps}$ should correspond to a single individual in
stochastic modeling.  Thus, if we come back to the starting point,
which means a total population $\int n(x,t) \, dx$ of
$10^{10}-{10}^{12}$, then for $\eps=10^{-4}$ (a typical value for many
applications) and threshold $\bar{n}_{\eps}$ of order
$\exp(-\frac{1}{\eps})$, an aggregate population over any fixed
interval of traits would still represent much less than one
individual.

\par

Another type of correction has been proposed by Jabin \cite{Jabin12}.
The author allows the threshold $\bar{n}_{\eps}$ to be polynomial in
$\eps$ and introduces special cooperative term  $D_{\eps}$ in the (rescaled) model
\begin{equation}
\label{PEmodel}
  \del_t n_{\eps}(x,t) = \Big( r(x)-\int I(x-y)n(t,y) - D_{\eps}[n] \Big)n(x,t) +M[n](t,x).
\end{equation}
This term does not handle the small populations as precisely, but the new
model still corrects all the abnormal behaviors of \eqref{ctraitmut}
near the limit.
% The correction is inspired from the case with sexual
% reproduction. In the present context of mostly asexual reproduction,
% the effect is better understood as taking into account some
% cooperative effects between the individuals. Jabin \cite{Jabin12}
% studies the models of the type
% \begin{equation}\label{ctraitmut4}
%  \del_t n_{\eps}(x,t)  = \frac{1}{\eps}\Big( r(x)-I \star n_{\eps} - D_{\eps}[n_{\eps}](x,t)\Big) n_{\eps} -\frac{1}{\eps}\sqrt{\frac{n_{\eps}}{\bar{n}_{\eps}}}+M_{\eps}[n](t,x)
%\end{equation}
The cooperative effects in \cite{Jabin12} were, however, more intuited
than derived.  For example, the typical cooperative term $D_{\eps}$
has the form
\begin{equation}\label{coopt}
  -D_{\eps}[n_{\eps}](x,t) = -D_0 + \max\big(0,\,D_0-K(x) \SP d(x,
      \{n_{\eps} \geq \bar{n}_{\eps})\big)
\end{equation}
where $K(x)$ is a symmetric positive kernel.  In that respect, the
present work puts the approach in \cite{Jabin12} on a more solid
framework by actually deriving those effects from realistic
biochemical processes.

% In addition it is for the moment the only correction for which one
% can derive rigorously the limit. That means in particular that one
% can obtain numerical simulations for realistically low values of
% $\eps$.

\par

The present work aims at introducing cooperative terms, similar to
those of \cite{Jabin12}, that arise naturally (directly from biological
processes), rather than {\it ad hoc} mathematical terms. The
cooperative effects in the integral operator $B[n]$ in \eqref{ctrait}
appear naturally in the process of model construction and give a hint
of what such terms should look like.

\par\medskip

{\bf Acknowledgments.} P.E. Jabin was partially supported by NSF Grant 1312142, NSF Grant 1614537, and by NSF Grant RNMS (Ki-Net) 1107444.


\begin{thebibliography}{99}

\setlength{\parskip}{1em}

\bibitem{Barles07}
G. Barles and  B. Perthame, Concentrations and constrained Hamilton-Jacobi equations arising in adaptive dynamics. {\em Contemp. Math.} (2007), 439, 57-68, Amer. Math. Soc., Providence, RI.

\bibitem{Burger96}
R. Burger and I.M. Bomze, Stationary distributions under mutation-selection balance: structure and properties. {\em Adv. Appl. Prob.} (1996), 28, 227-251.


\bibitem{Calcina04}
\'{A}. Calcina and S. Cuadrado, Small mutation rate and evolutionarily stable
strategies in infinite dimensional adaptive dynamics. {\it J. Math. Biol.} (2004), 48, 135–159.

\bibitem{Champagnat01}
N. Champagnat, R. Ferri\'{e}re and G. Ben Arous. The canonical equation of adaptive dynamics: a mathematical view. {\em Selection} (2001), 2, 71-81.

\bibitem{Champagnat06-1}
N. Champagnat, A microscopic interpretation for adaptive dynamics trait substitution sequence models. {\it Stoch. Process. Appl.} (2006), 116, 1127-60.

\bibitem{Champagnat06-2}
N. Champagnat, R. Ferri\'{e}re, S. M\'{e}l\'{e}ard, Unifying evolutionary dynamics: from individual stochastic processes to macroscopic models. {\it Theor. Popul. Biol.} (2006), 69, 297-321.


\bibitem{Champagnat08}
N. Champagnat, R. Ferri\'{e}re, S. M\'{e}l\'{e}ard. From individual stochastic processes to macroscopic models in adaptive evolution, {\em Stoch. Models} (2008), 24, suppl. 1, 2-44.


\bibitem{Champagnat10}
N. Champagnat, P.-E. Jabin and G. Raoul. Convergence to equilibrium in competitive Lotka-Volterra and chemostat systems, {\em C. R. Math. Acad. Sci. Paris} (2010), 348, 1267-72.


\bibitem{Desvillettes08}
L. Desvillettes, P.-E. Jabin, S. Mischler and G. Raoul. On selection dynamics for continuous structured populations. {\em Commun. Math. Sci.} (2008), 6(3), 729-747.


\bibitem{Dieckmann96}
U. Dieckmann and R. Law, The dynamical theory of coevolution: a derivation from stochastic ecological processes. {\em J. Math. Biol.} (1996), 34, 579-612.


\bibitem{ODieckmann01}
O. Diekmann, M. Gyllenberg, H. Huang, M. Kirkilionis, J.A.J. Metz and H.R. Thieme. On the formulation and analysis of general deterministic structured population models. II. Nonlinear theory. {\em J. Math. Biol.} (2001), 43, 157-189.

\bibitem{ODieckmann04}
O. Diekmann, A beginner's guide to adaptive dynamics. {\em Mathematical modelling of population dynamics, Banach Center Publ.} (2004), 63, 47-86, Polish Acad. Sci., Warsaw.

\bibitem{ODieckmann05}
O. Diekmann, P.E. Jabin, S. Mischler and B. Perthame, The dynamics of adaptation: An illuminating example and a Hamilton-Jacobi approach. {\em Theor. Popul. Biol.} (2005) 67, 257-271.


\bibitem{Ferriere02}
R. Ferri\`{e}re, J.L. Bronstein, S. Rinaldi, R. Law and M. Gauduchon (2002). Cheating and the evolutionary stability of mutualisms. {\em Proceedings of the Royal Society of London B} (2002), 269:773-780.

\bibitem{Galliard03}
J.F. Le Galliard, R. Ferri\`ere and U. Dieckmann, The adaptive dynamics of altruism in spatially heterogeneous populations. {\em Evolution} (2003) 57: 1-17.


\bibitem{Ghosal03}
S. Ghosal and S. Mandre, A simple model illustrating the role of turbulent
life on phytoplankton blooms. {\it J. Math. Biol.} (2003), 46, 4, 333–346.

\bibitem{Gopalsamy84}
K. Gopalsamy, Global asymptotic stability in Volterra's population systems. {\em J. Math. Biology} (1984), 19, 157-168.


\bibitem{Gyllenberg}
M. Gyllenberg and G. Mesz\'{e}na, On the impossibility of coexistence of infinitely many strategies, {J. Math. Biol.} (2005), 50 133-160.

\bibitem{Hirsch88}
M.W. Hirsch, Systems of differential equations which are competitive or cooperative. III. Competing species. Nonlinearity 1(1), 51-71 (1988).

\bibitem{Hofbauer98}
J. Hofbauer and K. Sigmund, {\em Evolutionary Games and Population Dynamics}, Cambridge University Press, Cambridge (1998).

\bibitem{Jabin11}
P.E. Jabin and G. Raoul, Selection dynamics with competition. {\em J. Math. Biol.} (2011), 63, issue 3, 493-517.

\bibitem{Jabin12}
P.E. Jabin, Small populations corrections for selection-mutation models, {\em Netw. Heterog. Media} (2012), 7, issue 4, 805-836.


\bibitem{Jones09}
E.I. Jones, R. Ferri\'{e}re, J.L. Bronstein, Eco-evolutionary dynamics of mutualists and exploiters. {\em The American Naturalist} (2009), 174: 780-794.

\bibitem{Jones12}
E.I. Jones, J.L. Bronstein and R. Ferri\`{e}re, The fundamental role of competition in the ecology and evolution of mutualisms. \emph{Ann NY Acad Sci.} (2012),  1256:66-88.

\bibitem{Perthame07}
{B. Perthame}, {\em Transport Equations in Biology}, Birkhouser Verlag (2007).

\bibitem{Perthame09}
B. Perthame and M. Gauduchon, Survival thresholds and mortality rates in adaptive dynamics: conciliating deterministic and stochastic simulations. {\em IMA Journal of Mathematical Medicine and Biology} (2009).

\bibitem{Smith95}
H.Smith and P. Waltman. {\em The Theory of the Chemostat. Dynamics of Microbial Competition}. Cambridge University Press (1995).

\bibitem{Holland10}
J.N. Holland, D.L. DeAngelis. A consumer-resource approach to the density-dependent population dynamics of mutualism. {\em Ecology} (2010), 91(5):1286-95.

\bibitem{Kiers10}
E.T. Kiers, T.M. Palmer, A.R. Ives, J.F. Bruno, and J.L. Bronstein,  Mutualisms in a changing world: an evolutionary perspective. {\em Ecology Letters} (2010), 13:1459-1474.

\bibitem{Krisztina08}
K. Krisztina and S. Kov\'{a}cs Qualitative behavior of n-dimensional ratio-dependent predator-prey systems. {\em Appl. Math. Comput.} (2008), 199(2), 535-546.

\bibitem{LESH95}
S.B. Leshine, Cellulose degradation in anaerobic environments. {\em Annu. Rev. Microb.} (1995), 49:399-426.

\bibitem{LYND02}
L. R. Lynd, P. J. Weimer,  W. H. van Zyl, I.S. Pretorius,
Microbial Cellulose Utilization: Fundamentals and Biotechnology, {\em Microbiology and molecular biology reviews.} (2002), 66-3, 506-577.

\bibitem{LWL95}
G.F. Lawler, {\em Introduction to Stochastic Processes.} Chapman \& Hall/CRC (1995).


\bibitem{Meleard09}
S. M\'{e}l\'{e}ard, Introduction to stochastic models for evolution. {\em Markov Process and Related Fields} (2009), 15,  issue 3, 259-264.

\bibitem{Metz92}
J.A.J. Metz, R.M. Nisbet, and S. A. H. Geritz, How should we define 'fitness' for general ecological scenarios? {\em Trends in Ecology and Evolution} (1992), 7, 198-202.


\bibitem{Metz96}
J.A.J. Metz, S.A.H. Geritz, G. Meszena, F.J.A. Jacobs and J.S. van Heerwaarden, Adaptive dynamics, a geometrical study of the consequences of nearly faithful reproduction. In: {\it Stochastic and spatial structures of dynamical systems} (Amsterdam, 1995), 183–231, Konink. Nederl. Akad. Wetensch. Verh. Afd. Natuurk. Eerste Reeks, 45, North-Holland, Amsterdam (1996).

\bibitem{Novick50}
Novick A. and L. Szilard, Experiments with the Chemostat on Spontaneous Mutations of Bacteria, {\em PNAS} (1950) 36, 708-719.

\bibitem{Platt09}
T.G. Platt and J.D. Bever. Kin competition and the evolution of cooperation. {\em Trends Ecol Evol.} (2009), 24(7):370-7.

\bibitem{Raoul11}
G. Raoul, Long time evolution of populations under selection and rare mutations, {\em Acta Applicandae Mathematica} (2011), 114, issue 1-2, 1-14.

\bibitem{TLKAR98}
H. M. Taylor, S. Karlin, {\em An introduction to Stochastic Modeling.} Academic Press (1998).

\bibitem{Zeeman93}
M.L. Zeeman, Hopf bifurcations in competitive three-dimensional Lotka-Volterra systems. {\em Dyn. Stab. Syst.} (1993), 8(3), 189–217.


\end{thebibliography}
\end{document}